\documentclass[11pt]{amsart}
\usepackage{amssymb}   
\usepackage{amsthm}
\usepackage{mathtools}
\usepackage{color}
\usepackage{amsmath}
\usepackage{graphicx}
\usepackage[show]{ed} 
\usepackage[T1]{fontenc}
\usepackage[english]{babel}
\usepackage{enumerate}
\usepackage{babel}
\usepackage[left=3cm,right=3cm,top=3cm,bottom=3cm]{geometry}
\usepackage[pdftex,breaklinks,colorlinks, citecolor=blue,urlcolor=blue]{hyperref}
\numberwithin{equation}{section}

\newtheorem{main}{Theorem}

\newtheorem{theorem}{Theorem}[section]

\newtheorem{theo}{Theorem}[section] 
\newtheorem{lem}[theo]{Lemma}
\newtheorem{corll}[theo]{Corollary}
\newtheorem{rem}[theo]{Remark}
\newtheorem{clm}[theo]{Claim}
\newtheorem{prop}[theo]{Proposition}

\usepackage{relsize}

\newcommand{\R}{\mathbb{R}}
\newcommand{\C}{\mathbb{C}}
\newcommand{\N}{\mathbb{N}}
\newcommand{\Z}{\mathbb{Z}}

\DeclareMathOperator{\re}{Re}
\DeclareMathOperator{\im}{Im}

\def \intb {\int_{\partial \Omega}}
\def \intom {\int_{\Omega}}
\def \intr {\int_{\R^3}}
\def \ds  {\delta^{*}}
\def \dt  {\tilde{\delta}}
\def \uu  {\underline{u}}
\def \vv  {\underline{v}}
\def \bh  {\underline{h}}
\def \th  {\underline{h}_{_X}}
\def \thre  {\underline{h}_{1{_X}}}
\def \thima  {\underline{h}_{2{_X}}}
\def \esp {    L^5 H^{1,\frac{30}{11}}(\R \times \R^3)       }

\def \e  {\frac{ e^{- \left| X(t) \right|}}{\left| X(t) \right|} } 
\def \ee {\frac{ e^{-2 \left| X(t) \right|}}{\left| X(t) \right|^2 } } 
\def \eee {\frac{ e^{-3 \left| X(t) \right|}}{\left| X(t) \right|^3 } } 
\def \eeee {\frac{ e^{-4 \left| X(t) \right|}}{\left| X(t) \right|^4 } }

\def \nn   {\vec{n}}

\def \Nls {{\rm{(NLS$_\Omega$) }}}
\def \NNls {{\rm{NLS$_\Omega$ }}}
\def \nls {{\rm{(NLS$_{\R^3}$) }}}
\def \nnls {{\rm{NLS$_{\R^3}$ }}}
\def \hh   {\widetilde{h}}
\title[3D  NLS outside a strictly convex obstacle]{Threshold solutions in the focusing 3D cubic NLS equation outside a strictly convex obstacle}

\author{Thomas Duyckaerts }
\email[]{duyckaer@univ-paris13.fr}
\address{LAGA, UMR 7539, Institut Galil\'ee, Universit\'e Sorbonne Paris Nord and Institut Universitaire de France}

\author{Oussama Landoulsi}

\email[]{landoulsi@univ-paris13.fr}
\address{LAGA, UMR 7539, Institut Galil\'ee, Universit\'e Sorbonne Paris Nord}

\author{Svetlana Roudenko }
\email[]{sroudenko@fiu.edu}
\address{Department of Mathematics \&  Statistics, Florida International University}

\begin{document}

\maketitle
\begin{abstract}
We study the dynamics of the focusing $3d$ cubic nonlinear Schr\"odinger equation in the exterior of a strictly convex obstacle at the mass-energy threshold, namely, when $E_{\Omega}[u_0] M_{\Omega}[u_0] = E_{\R^3}[Q] M_{\R^3}[Q]$ and $\left\| \nabla u_0 \right\|_{L^{2}(\Omega)} \left\| u_0 \right\|_{L^{2}(\Omega)}< \left\| \nabla Q \right\|_{L^2(\R^3)} \left\| Q \right\|_{L^2(\R^3)} ,$ where $u_0 \in H^1_0(\Omega)$ is the initial data, $Q$ is the ground state on the Euclidean space, $E$ is the energy and $M$ is the mass. In the whole Euclidean space Duyckaerts and Roudenko (following the work of Duyckaerts and Merle on the energy-critical problem) have proved  the existence of a specific global solution that scatters for negative times and converges to the soliton in positive times. We prove that these heteroclinic orbits do not exist for the problem in the exterior domain and that all solutions  at the threshold are globally defined and scatter. The main difficulty is the control of the space translation parameter, since the Galilean transformation is not available.  
\end{abstract}

{
  \hypersetup{linkcolor=black}
  \tableofcontents
}

\section{Introduction}
We consider the focusing nonlinear Schr\"odinger equation in the exterior of a smooth compact strictly convex obstacle $\Theta \subset \R^3$ with Dirichlet boundary conditions: 
\begin{equation}
\tag{NLS$_{\Omega}$}
\begin{cases}
\label{NLS} 		
i\partial_tu+\Delta_{\Omega} u = -|u|^{2}u  \qquad  &  (t,x)\in \R \times\Omega ,\\ 
u(t_0,x) =u_0(x)  &  x \in \Omega  , \\
u(t,x)=0  &   (t,x)\in \R \times \partial\Omega ,
\end{cases}		
\end{equation}		
     where $\Omega=\R^3 \setminus \Theta$, $\Delta_{\Omega}$ is the Dirichlet Laplace operator on $\Omega$ and $t_0 \in \R$ is the initial time. Here, $u$ is a complex-valued function,
\begin{align*}
 u:  \R &\times  \Omega \longrightarrow \C \\
 	(t &, x)  \longmapsto  u(t,x) .
 \end{align*}
   We take the initial data $u_0 \in H^1_0(\Omega),$ where $H^1_0(\Omega)$ is the Sobolev space 
   $$ \{ u \in L^2(\Omega) \; \text{such that} \; \left| \nabla u \right| \in L^2(\Omega) \; \text{and} \; u_{|\partial \Omega}=0\}. $$
    
    The \NNls equation is locally wellposed in $H^1_0(\Omega),$ see \cite{AnRa08}, \cite{PlVe09}, \cite{Ivanovici10} and \cite{BlairSmithSogge2012}.  The solutions of the \NNls equation satisfy the mass and energy conservation laws: 
 \begin{align*}
 M_{_\Omega}[u(t)] &:= \int_{\Omega} |u(t,x)|^2 dx = M[u_0] ,\\
E_{\Omega}[u(t)] &:= \frac{1}{2} \int_{\Omega} |\nabla u(t,x)|^2 dx - \frac{1}{4} \int_{\Omega}\left|u(t,x)\right|^{4}dx=E[u_0].  
\end{align*} 
 
 Unlike the nonlinear Schr\"odinger equation \nnls posed on the whole Euclidean space $\R^3,$ the \NNls equation does not have the momentum conservation. \\
 
The \nnls equation is invariant by the scaling transformation, that is, 
 \begin{equation*}
  u(t,x) \longmapsto \lambda u(\lambda x,\lambda^2 t) \;  \; \text{ for }  \lambda >0.
 \end{equation*}
 This scaling identifies the critical Sobolev space $\dot{H}^{\frac 12 }_{x}$. Since the presence of an obstacle does not change
the intrinsic dimensionality of the problem, we may regard the \NNls equation as being $H^1(\Omega)$-subcritical and $L^2(\Omega)$-supercritical.  \\ 

In this paper,  we study the global well-posedness and scattering of solutions to the \NNls equation. We start recalling earlier results on global existence and scattering (\cite{PlVe09}, \cite{KiVisnaZhang16}): if $u$ has a finite Strichartz norm (Cf. Theorem \ref{Strichartz}), then $u$ scatters in $H^1_0(\Omega),$ i.e., 
\begin{align*}
    \exists  u_{\pm} \in H^1_0(\Omega) \quad \text{ such that } \quad  \lim_{t \longrightarrow  \pm \infty } \left\| u(t)-e^{it \Delta_\Omega} u_{\pm} \right\|_{H^1_0(\Omega)}=0.
    \end{align*} 
If initial data are sufficiently small in $H^1_0(\Omega),$ then the corresponding solutions $u(t)$ are global and scatter in both time directions. \\

Global existence and scattering for large data was studied for the \nnls equation, posed on the whole Euclidean space $\R^3,$ in several articles in different contexts. The \nnls equation
has solutions of the form $e^{it \Delta_{\R^3}} Q,$ where $Q$ solves the following nonlinear elliptic equation 
 \begin{equation}
 \label{eq_Q}
 \begin{cases}
 -  \, Q +  \Delta Q +  \left| Q   \right|^{2} Q =0 ,   \\
 \;  Q   \in H^1(\R^3).
 \end{cases}
 \end{equation} 
 
In this paper, we denote by $Q$ the ground state solution, that is, the unique radial, vanishing at infinity, positive solution of \eqref{eq_Q}. Such $Q$ is smooth, exponentially decaying at infinity, and characterized as the unique minimizer for the Gagliardo-Nirenberg inequality up to scaling, space translation and phase shift, see \cite{Kwong89}. \\

In \cite{HoRo08}, the authors have studied the behavior (i.e., scattering and global existence)\footnote{also, blow-up, however, we do not need it in this paper} of the radial solutions of the focusing cubic \nnls equation on $\R^3$, when the initial data satisfies a criterion given by the ground state threshold, that is, if $u_0 \in H^1(\R^3),$ radial and satisfies  $M_{\R^3}[u_0]E_{\R^3}[u_0]<M_{\R^3}[Q]E_{\R^3}[Q]$ and $ \left\| \nabla u_0 \right\|_{L^2(\R^3)} \left\| u_0 \right\|_{L^2(\R^3)} < \left\| \nabla Q \right\|_{L^2(\R^3)} \left\| Q \right\|_{L^2(\R^3)},$ then $u$ scatters in $H^1(\R^3).$ Note that, this criterion is expressed in terms of the scale-invariant quantities $\left\| \nabla u_0 \right\|_{L^2} \left\| u_0 \right\|_{L^2}$ and $M[u_0]E[u_0].$ This result was later extended to the non-radial case in \cite{DuHoRo08} and
to arbitrary space dimensions and focusing intercritical power nonlinearities in \cite{FaXiCa11} and \cite{Guevara14}. \\ 

The question about the global existence and scattering for the focusing cubic \NNls outside of a strictly convex obstacle was studied in \cite{KiVisnaZhang16}, which we state next.
\begin{theorem}
\label{theo-killip-all}
Let $u_0 \in H^1_0(\Omega)$ satisfy
\begin{align}
\label{u<Qthro}
    \left\| u_0 \right\|_{L^2\footnotesize{\left( \Omega \right)}}\left\| \nabla u_0 \right\|_{L^2\footnotesize{\left( \Omega \right)}}  & < \left\| Q  \right\|_{L^2\footnotesize{\left(\R^3 \right)}} \left\| \nabla Q \right\|_{L^2\footnotesize{\left(\R^3 \right)}} , \\
    \label{MEuu=MEQQTheo}
    M_{\Omega}[u_0]E_{\Omega}[u_0]& < M_{\R^3}[Q]  E_{\R^3}[Q].
\end{align}
Then $u$ scatters in $H^1_0(\Omega),$ in both time directions.
\end{theorem}

The purpose of this paper is to study the behavior of solutions to the \NNls equation exactly at the mass-energy threshold, i.e., when \begin{equation}
\label{ME_u=ME_Q}
E_{\Omega}[u]M_{\Omega}[u]=E_{\R^3}[Q]M_{\R^3}[Q].
\end{equation} 
In \cite{DuRo10} T.\,Duyckaerts and S.\,Roudenko described the behavior of the solutions of the \nnls equation at the mass-energy threshold. At this mass-energy level, the \nnls equation has a richer dynamics for the long time behavior of the solutions compared to the result mentioned above. The authors proved the existence of special solutions, denoted by $Q^{+}$ and $Q^{-}.$ These solutions approach the soliton, up to symmetries, in one time direction: there exits $e_0>0$ such that 
\begin{equation}
\label{appro-Q}
  \left\|  Q^{\pm} -e^{it}Q \right\|_{H^1(\R^3)} \leq c e^{-e_0 t}  \quad \mbox{for}\;   t \geq 0.
  \end{equation}
The behavior of $Q^{\pm}$ in the opposite time direction is completely different: $Q^{-}$ scatters for negative time but $Q^{+}$ has a finite time of existence. The existence of these special solutions is derived from the existence of the two real nonzero eigenvalues for the linearized operator around the soliton $e^{it} Q.$ These special solutions have the same mass-energy of the soliton, $M_{\R^3}[Q^{\pm}]E_{\R^3}[Q^{\pm}]=M_{\R^3}[Q]E_{\R^3}[Q]$, however, $\left\| \nabla Q^{-} (t) \right\|_{L^2(\R^3)} < \left\| \nabla Q \right\|_{L^2(\R^3)}$ and $\left\| \nabla Q^{+}(t) \right\|_{L^2(\R^3)} > \left\| \nabla Q \right\|_{L^2(\R^3)},$ for all $t$ in the interval of existence of $Q^{\pm}.$ Furthermore, if we consider initial data $u_0 $  such that \eqref{ME_u=ME_Q} holds and $\left\| u_0 \right\|_{L^2(\R^3)} \left\| \nabla u_0 \right\|_{L^2(\R^3)}  < \left\| Q \right\|_{L^2(\R^3)} \left\| \nabla Q \right\|_{L^2(\R^3)}, $ respectively, $\left\| u_0 \right\|_{L^2(\R^3)} \left\| \nabla u_0 \right\|_{L^2(\R^3)} > \left\| Q \right\|_{L^2(\R^3)} \left\| \nabla Q \right\|_{L^2(\R^3)} ,$ then the corresponding solution $u(t)$ of the \nnls equation is global and either (i) scatters in $H^1(\R^3)$ or $u \equiv Q^{-}$, up to the symmetries, or (ii) $u(t)$ has a finite time of existence or $u \equiv Q^{+}$, up to the symmetries. Furthermore, if the gradient of $u_0$ is equal to the gradient of $Q,$ then $u(t)$ is equal to $Q,$ up to the symmetries.  \\

Note that for the \NNls equation, there do not exist analogs of the solutions $e^{it}Q$, $Q^{\pm}$ at the threshold $M_{\Omega}[u]E_{\Omega}[u]=M_{\R^3}[Q]E_{\R^3}[Q].$ Indeed there is no function $u_0 \in H^1_0(\Omega) $  satisfying \eqref{ME_u=ME_Q} and $ \left\| \nabla u_0 \right\|_{L^2(\Omega)} \left\| u_0 \right\|_{L^2(\Omega)} =  \left\| \nabla Q \right\|_{L^2(\R^3)} \left\| Q \right\|_{L^2(\R^3)} .$ By extending $u_0$ with $0$ on the obstacle, the solution $u_0$ must be equal to $Q,$ up to the symmetries, which would not satisfy Dirichlet boundary conditions.  Similarly, in the presence of the obstacle there is no function in $H^1_0(\Omega)$ such that \eqref{appro-Q} holds, since such a solution has to converge to $Q$  for the sequence of times $t_n=2\pi n$, contradicting the fact that $Q$ does not satisfy Dirichlet boundary conditions.  \\

We now state the main result of this paper. 
\begin{main}
\label{main-theo}
Let $u_0 \in H^1_0(\Omega)$ and let $u(t)$ be the corresponding solution to \Nls such that $u_0$ satisfy 
\begin{equation}
\label{condition-theo-prin}
M_{\Omega}[u]E_{\Omega}[u]=M_{\R^3}[Q]E_{\R^3}[Q] \quad  \text{ and } \left\| u_0 \right\|_{L^2( \Omega )} \left\| \nabla u_0 \right\|_{L^2(\Omega) } < \left\| Q  \right\|_{L^2{\left(\R^3 \right)}} \left\| \nabla Q \right\|_{L^2{\left(\R^3 \right)}}. 
\end{equation}
Then  $u$ scatters in $H^1_0(\Omega)$ in both time directions. 
\end{main}

The proof of this theorem is based on the approach of the Euclidean setting results in \cite{DuMe09a} and \cite{DuRo10}. The first step is similar to the proof of the compactness of the critical solution developed by 
C.\,Kenig and F.\, Merle in \cite{KeMe06} in the energy-critical setting and adapted to the energy-subcritical case  in \cite{HoRo08} and \cite{DuHoRo08}. It uses a concentration-compactness argument that requires a profile decomposition as in the works of F.~Merle and L.~Vega \cite{MeVe98}, P.\,G{\'e}rard  \cite{Gerard98}, and S.\,Keraani \cite{Keraani01},
adapted by R.\,Killip, M.\,Visan and X.\,Zhang for the problem in the exterior of a convex obstacle in \cite{KiVisnaZhang16} (for energy-critical) and in \cite{KiVisZha16} (for the energy-subcritical).  
The second step of the proof is a careful study of the space translation and phase parameters for a solution of \NNls that is close to $Q,$ up to the transformations. In \cite{DuRo10}, the translation parameter is controlled using conservation of the momentum, leading ultimately to the fact that $e^{-it}u$ converges exponentially to $Q$. This conservation law is not available for the \NNls equation, and we must achieve this control through a new intricate limiting argument, that relies among other things on the uniqueness theorem in \cite{DuHoRo08}. \\

In \cite{Ouss20}, the second author has proved that when the obstacle is the Euclidean ball of $\R^3$, solutions such that 
$M_{\Omega}[u]E_{\Omega}[u]<M_{\R^3}[Q]E_{\R^3}[Q]$ and  $\left\| u_0 \right\|_{L^2( \Omega )} \left\| \nabla u_0 \right\|_{L^2(\Omega) } > \left\| Q  \right\|_{L^2{\left(\R^3 \right)}} \left\| \nabla Q \right\|_{L^2{\left(\R^3 \right)}}$ with a finite variance and a certain symmetry blow up in finite time. 
In view of the known results on $\R^3$, one should expect blow-up in finite or infinite time for all solutions of this type, however, the blow-up for the \NNls equation is a delicate issue. One of the difficulties is the appearance of boundary terms with the wrong sign in the virial identity that is used to prove blow-up on $\R^3$.   Blow-up is also expected in the threshold case $M_{\Omega}[u]E_{\Omega}[u]=M_{\R^3}[Q]E_{\R^3}[Q]$ and $\left\| u_0 \right\|_{L^2( \Omega )} \left\| \nabla u_0 \right\|_{L^2(\Omega) } >\left\| Q  \right\|_{L^2{\left(\R^3 \right)}} \left\| \nabla Q \right\|_{L^2{\left(\R^3 \right)}}$, which is an open question. Let us mention however that linear scattering is precluded for these solutions, since they satisfy the bound $\left\| u(t) \right\|_{L^2( \Omega )} \left\| \nabla u(t) \right\|_{L^2(\Omega) } > \left\| Q  \right\|_{L^2{\left(\R^3 \right)}} \left\| \nabla Q \right\|_{L^2{\left(\R^3 \right)}}$ on their domains of existence, and the condition on the energy.  \\

When $\Omega=\R^3$, K.~Nakanishi and W.~Schlag \cite{NakanishiSchlag12} described the dynamics of solutions slightly above the mass-energy threshold, that is such that $E_{\R^3}[Q]_{\R^3}M[Q]\leq E_{\R^3}[u_0]M_{\R^3}[u_0]<E_{\R^3}[Q]_{\R^3}M[Q]+\varepsilon$ for a small $\varepsilon>0$, showing that all $9$ expected behaviors (any combination of blow-up in finite time, linear scattering or scattering to the ground state solution) do indeed occur. Some sufficient conditions for scattering and blow-up in this regime are given by the first and third authors in \cite{DuRo15}. The analog of the result in \cite{NakanishiSchlag12} outside of an obstacle is currently out of reach, due to insufficient understanding of blow-up in finite time. Let us mention however that in this case, the soliton-like behavior is possible. Indeed, the second author in \cite{Ou19} constructed a solution behaving as a travelling wave in $\R^3$ for large $t,$ moving away from the obstacle with an arbitrary small speed $v$ and such that $E[u_0]M[u_0]=E[Q]M[Q]+c |v|^2$ for a constant $c>0$. See also \cite{OuKaiSvetTh20} for numerical investigations in this regime. \\

The study of the influence of the underlying space geometry on the dynamics of the wave-type equations started long ago. Let us mention some of the works 
on a wave-type equation in the exterior of an obstacle with Dirichlet or Neuman boundary conditions. In $1959,$ H.~W.~Calvin studied  the rate of decay of solutions to the linear wave equation outside of a sphere, see \cite{Wilcox59}. Later, Morawetz extended this result to a star-shape obstacle, see \cite{Morawetz61}, \cite{Morawetz62} and also \cite{LaxMorawPhillip62}, \cite{LaxMorawPhillip63}. After that, various results were obtained for an almost-star shape, non-trapping and moving obstacle, see \cite{Ivri69},
\cite{MoraTalstonStrauss77}, \cite{MoraRalstonStraussCorec78} and  \cite{CooperStrauss76}. 
The Cauchy theory for the \NNls equation with initial data in $H^1_0(\Omega),$ was initiated in $2004$ by N.\,Burq, P.\,Gérard and N.\,Tzvetkov in  \cite{BuGeTz04a}. Assuming that the obstacle is non-trapping, the authors proved a local existence result for the $3d$ sub-cubic (i.e., $p<3$) \NNls equation. This was later extended by R.\,Anton in \cite{AnRa08} for the cubic nonlinearity, by F.\,Planchon and L.\,Vega in \cite{PlVe09} for the energy-subcritical \NNls equation in dimension $d=3$ (i.e., $1<p<5$) and by F.\,Planchon and O.\,Ivanovici in \cite{MR2683754} for the energy-critical case in dimension $d=3$ (i.e, $p=5$), see also  \cite{BlairSmithSogge2012} and \cite{Ivanovici07}, \cite{Ivanovici10}, \cite{DongSmithZhang2012} for convex obstacle. The local well-posedness in the critical Sobolev space was first obtained in \cite{MR2683754}, for $3+\frac 25< p< 5.$  In \cite{Ou19}, the second author extended this result for $\frac 73<p <5,$ using the fractional chain rule in the exterior of a compact convex obstacle from \cite{killip2015riesz}. \\

The paper is organized as follows: In Section \ref{Preliminaries}, we recall known properties of the ground state and coercivity property associated to the linearized operator under certain orthogonality conditions. There, we also recall Strichartz estimates, stability theory and the profile decomposition for the \NNls equation outside of a strictly convex obstacle. 
In Section \ref{Sec-Modulation}, we discuss modulation, in particular, in \S \ref{modulation} we use the modulation in phase rotation and in space translation parameters near the  truncated ground state solution, in order to obtain orthogonality conditions.  Section \ref{sec-scattering} is dedicated to the proof of the main theorem. In \S \ref{Compactness properties} we use the profile decomposition to prove a compactness property, which yields the existence of a continuous translation parameter $x(t)$ such that the extension of a non-scattering solution $\uu(t,x+x(t)),$ that satisfies \eqref{condition-theo-prin}, is compact in $H^1(\R^3). $ In \S \ref{sec-bddxx}, we control the space translation $x(t)$ by approximating it by auxiliary translation parameter given by modulation on $\R^3,$ in \cite{DuRo10}. Moreover, we use a local virial identity with estimates from previous sections on the modulation parameter to prove that $x(t)$ is bounded. In \S \ref{Sec-convergence-mean}, we prove that the parameter $\delta(t):= \left| \left\| \nabla Q \right\|_{L^2} -\left\| \nabla u \right\|_{L^2} \right|$ converges to $0$ in mean. 
Finally, we conclude the proof of Theorem \ref{main-theo} using the compactness properties with the control of the space translation parameter $x(t)$ and the convergence in mean.  
\smallskip

{\bf Acknowledgments.}
T.D. was partially supported by Institut Universitaire de France and Labex MME-DII. Part of the research on this project was done while O.L. was visiting the Department of Mathematics and Statistics at Florida International University, Miami, USA, during his PhD training. He thanks the department and the university for hospitality and support. S.R. was partially supported by the NSF grant DMS-1927258. Part of O.L.'s research visit to FIU was funded by the same grant DMS-1927258 (PI: Roudenko).  \\
\smallskip

\textbf{Notation.} 
Define $\Psi$ as a $C^{\infty}$ function such that
 \begin{equation}
\label{def-psi}
\Psi=\begin{cases}
0 & \; \text{near} \; \Theta ,\\
1 & \; \mbox{if} \;  \left|x\right| \gg 1 .
\end{cases}  
\end{equation}
We write $ a = O (b), $ when $a$ and $b$ are two quantities, and there exists a positive constant $C$  independent of parameters,  such that $\left| a  \right| \leq C \,  b ,$ and  $a \approx  b ,$ when $a =O(b)$ and $b=O(a).$
For $h \in \C $, we denote $h_1=\re h $ and $h_2=\im h $. 
Throughout this paper, $C$ denotes a large positive constant and $c$ is a small positive constant, that may change from line to line; both do not depend on parameters. We denote by $\left| \cdot \right|$ the Euclidean norm on $\R^3.$
For simplicity, we write $\Delta =\Delta_\Omega$. 
The real $L^2$-scalar product $(\cdot, \cdot )$ means
$$ 
(f,g)=\re{\int f \;  \overline{g}} = \int \re g \, \re f + \int \im g \,  \im f \; .  $$

\section{Preliminaries}
\label{Preliminaries}
\subsection{Properties of the ground state} We recall here some well-known properties of the ground state. We refer the reader to \cite{Weinstein82}, \cite{Kwong89} , \cite[Appendix B]{Tao06BO} for a general setting and \cite{HoRo08} for the $3d$ cubic \nnls case, for more details. Consider the following nonlinear elliptic equation on $\R^3$
\begin{equation}
\label {equQ}
-Q + \Delta Q + |Q|^2Q=0.
\end{equation}
We are interested in a positive, decaying at infinity, solution $Q \in H^1(\R^3).$ The ground state solution is the unique positive, radial, vanishing at infinity, 
smooth solution of \eqref{equQ}. It is also (up to standard transformations) the unique minimizer of the Gagliardo-Nirenberg inequality:  if $u \in H^1(\R^3),$ then
\begin{equation}
\label{Qgaglia}
\left\|  u \right\|_{L^4(\R^3)}^4 \leq C_{GN} \left\| \nabla u \right\|_{L^2(\R^3)}^3 \left\| u \right\|_{L^2(\R^3)}, \quad \left\|  Q \right\|_{L^4(\R^3)}^4 =C_{GN} \left\| \nabla Q \right\|_{L^2(\R^3)}^3 \left\| Q \right\|_{L^2(\R^3)}.
\end{equation}
Moreover,
\begin{multline}
\left\| u \right\|_{L^4(\R^3)}^4= C_{GN} \left\| \nabla u \right\|_{L^2(\R^3)}^3 \left\| u \right\|_{L^2(\R^3)} \\ \Longrightarrow \exists \lambda_0 \in \C ,\exists\mu_0\in \R, \exists x_0 \in \R^3: u(x)= \lambda_0  Q(\mu_0(x+x_0)).
\end{multline}
We also have the Pohozhaev identities

\begin{equation}
\left\| Q \right\|_{L^4(\R^3)}^4= 4 \left\| Q \right\|_{L^2(\R^3)}^2  \quad \text{ and } \quad  \left\| \nabla Q \right\|_{L^2(\R^3)}^2=3 \left\| Q \right\|_{L^2(\R^3)}^2. 
\end{equation}
As a consequence of \eqref{Qgaglia}, one has  
\begin{prop}
\label{uConvQ}
There exists a function $\varepsilon(\eta),$ defined for small $\eta >0,$  such that $\displaystyle \lim_{\eta \rightarrow 0} \varepsilon(\eta)=0 $ and 
\begin{multline}
\label{u-pro-Q}
\forall u \in H^1(\R^3), \quad \big| \left\| u \right\|_{L^4(\R^3)}- \left\| Q \right\|_{L^4(\R^3)} \big| + \left| \left\|  u  \right\|_{L^2(\R^3)}- \left\|  Q  \right\|_{L^2(\R^3)} \right| +\\   \left| \left\| \nabla u \right\|_{L^2(\R^3)}- \left\| \nabla Q \right\|_{L^2(\R^3)} \right| \leq \eta  
\Longrightarrow \exists \, \theta_0 \in \R \text{ and } \exists \,x_0 \in \R^3 :  \;  \left\| u-e^{i \theta_0} Q(\cdot-x_0) \right\|_{H^1(\R^3)} \leq \varepsilon(\eta).
\end{multline}
\end{prop}
Next, we recall some known properties on the decay of $Q$, see  \cite{GiNiNi81}, \cite{BeLi83a} and \cite[Chapter 8] {Cazenave03BO}.

\begin{prop}[Exponential decay of $Q$] 
\label{decay1 Q}
Let $Q$ be the ground state solution of \eqref{equQ}, then there exist $a,C>0$ such that for $|x|>1$, 
$$ \left|  Q(x)-\frac{a}{|x|}e^{-|x|}\right|\leq C \frac{e^{-|x|}}{|x|^{3/2}}. $$
Moreover, 
$$ \left| \nabla  Q(x)  + \nabla ^2 Q(x)  \right| \leq  C\frac{ e^{-|x|}}{|x|} . $$
\end{prop}

\begin{lem}
\label{compac-supp-argu}
Let $Q$ be the ground state solution of \eqref{equQ}, $M>0$ large, $ X \in \R^3$ and let $g $ be an $L^1$-function. Then for $k>0,$ we have
\begin{equation}
\label{eq-compac-supp-argu}
\left| X \right| \geq 2M  \;  \Longrightarrow \; 
\displaystyle  \int_{|x| \leq M }  \left( Q^k(x-X) +    \left| \nabla Q(x-X)  \right|^k \right)  g(x) \, dx  = O\left(  \frac{ e^{-k |X|}}{\left|X\right|^k} \right) ,   
\end{equation}
where $O(\cdot)$ depends on $k, \; g$ and $M.$  \\

Furthermore, there exists $c_M >0$ such that
\begin{equation}
\label{eq-compac-supp-argu-2}
\displaystyle  \int_{|x| \leq M }   Q^k(x-X)   \, dx   \geq c_M \,  \frac{ e^{-k |X|}}{\left|X\right|^k} .
\end{equation}
\end{lem}
\begin{proof}
First, note that
$$ \frac 12 \left| X \right| < \left| X \right| -M < \left| x-X\right|, \; \text{ and} \;    \left| X \right| \geq 2M. $$
This implies that, for $\left| X \right| \geq 2M $ we have 
$$  e^{- \left| x-X \right| } \leq e^{M} e^{- \left| X \right| } \quad  \text {and } \quad \frac{1}{2 \left| x-X\right|} \leq \frac{1}{ \left|X \right|}.$$
Using the exponential decay of $Q$ from Proposition \ref{decay1 Q}, we obtain, 
$$ \int_{|x| \leq M } Q^k(x-X)  g(x) \, dx  = O\left(  \frac{ e^{-k |X|}}{\left| X \right|^k } \right), \quad \text{ for }   \;  k>0.  $$ 
Similarly, we get  $$ \int_{|x| \leq M }\left| \nabla Q(x-X)  \right|^k g(x) \, dx  = O\left(  \frac{ e^{-k|X|}}{\left|X\right|^k} \right),  \quad \text{ for }   \;  k>0. $$
The proof of \eqref{eq-compac-supp-argu-2} is similar by applying again Proposition \ref{decay1 Q} and we omit it. 
\end{proof}
Let $u \in H^1_0(\Omega)$ and define  $\uu \in H^1(\R^3) $ such that 
\begin{equation}
\label{defunderlineU}
\underline{u}(x)=\begin{cases}
u(x) \; \;  \forall x \in \; \Omega,  \\
0 \quad \quad  \forall x \in   \; \Omega^c. 
\end{cases}
\end{equation}
\begin{rem} 

We denote by $M_{\R^3}[\uu]=\left\| \uu \right\|^2_{L^2(\R^3)}$ and $E_{\R^3}[\uu]=\frac 12 \left\| \nabla \uu \right\|_{L^2(\R^3)}^2- \frac{1}{p+1} \left\| \uu \right\|_{L^{p+1}(\R^3)}^{p+1}.$ Note that, we have  $M_{\Omega}[u]=M_{\R^3}[\uu]$ and $E_{\Omega}[u]=E_{\R^3}[\uu].$  To simplify notations in what follows we drop the index $\Omega$ in the mass and the energy of the \NNls equation, so that we just write $M[u]$ and $E[u]$ instead of $M_{\Omega}[u]$ and $E_{\Omega}[u].$\\ 
\end{rem}

Assume that  $\underline{u}$ satisfies the left-hand side of \eqref{u-pro-Q}.
Then there exists $x_0 \in \R^3$ and $ \theta_0 \in \R$ such that   $$\left\| \uu-e^{i \theta_0} Q(\cdot-x_0) \right\|_{H^1 (\R^3)} \leq \varepsilon(\eta),$$ \\
which yields, by Proposition \ref{decay1 Q} and \eqref{eq-compac-supp-argu-2},
\begin{equation}
\label{x0large}
 \frac 1C \, \frac{e^{-|x_0|}}{|x_0|}  \leq \left\| Q(x-x_0) \right\|_{H^1(\Omega^c)} \leq \varepsilon(\eta) .
\end{equation}
This implies that $|x_0|$ is large when $\eta $ is small. 
\subsection{Coercivity  property}
We next recall some known properties of the linearized operator on $\R^3$. Consider a solution $u$ of \nnls close to $e^{it}Q$ and  write $u(t)$ as 
 $$ u(t,x)=e^{it}\left(Q(x)+\hbar(t,x)\right).$$
Note that $\hbar$ is the solution of the equation $$\partial_t \hbar + \mathcal{L}\hbar=\mathcal{R}(\hbar), \qquad   \mathcal{L} \hbar=-\mathcal{L}_{-} \hbar_2+i \mathcal{L}_{+} \hbar_1,$$ 
where   \begin{align*}
  \mathcal{L}_{+} \hbar_1 &=-\Delta \hbar_1+  \hbar_1-3 Q^2 \hbar_1,\\   \mathcal{L}_{-} \hbar_1&=-\Delta \hbar_2+ \hbar_2-Q^2 \hbar_2 ,\\   \mathcal{R}(\hbar)&= iQ (2|\hbar|^2+\hbar^2)+i|\hbar|^2\hbar.
\end{align*}  
Define $\Phi(\hbar)$, a linearized energy on $\R^3$, by 
\begin{equation}
\label{phi_R^3}
\Phi(\hbar):= \frac 12 \intr | \hbar|^2 +  \intr \frac 12 |\nabla \hbar |^2 - \frac 12 \intr Q^2 (3 \hbar_1^2+\hbar_2^2).
\end{equation}

We next define a subspace of  $H^1(\R^3)$, on which $\Phi$ is positive 
\begin{equation*}
\mathcal{G} := \bigg\{ \hbar \in H^1(\R^3) \;  \vert  \intr \partial_{x_j} Q \hbar_1 = 0, \intr Q \hbar_2 =0 , \, j=1,2,3   \bigg\}.
\end{equation*}
Then by \cite{DuRo10}, there exists $c>0$ such that 
\begin{equation}
\label{Coervicity-R}
\forall \hbar \in  \mathcal{G}, \quad \Phi(\hbar) \geq c \left\| \hbar \right\|_{H^1(\R^3)}^2.
\end{equation}
Let $h \in H^1(\R^3).$ Define \begin{equation}
\label{Quad-form-Phi}
\Phi_{X}(h):= \frac 12 \intr | \nabla h |^2 - \frac 12 \intr Q^2 \Psi^2(\cdot+X) (3 h_1^2+h_2^2)+ \frac{1}{2} \intr |h|^2,
\end{equation}
where $\Psi$ is defined in \eqref{def-psi}.
\begin{lem}
\label{lem-coercivity-Psi}
There exist $c>0$ such that for all $h \in H^1(\R^3)$ the following orthogonality relations hold for all $X \in \R^3$ with $|X|$ large
\begin{align}
\re \intr \Delta ( Q(x)& \Psi(x+X)) h(x+X) \, dx =0, \quad   \im \intr Q(x) \Psi(x+X)  h(x+X)\,  dx= 0,  \\
  & \re \intr \partial_{x_k}(Q(x) \Psi(x+X)) h(x+X) \,dx =0, \quad k=1,2,3.  
\end{align}
Then  \begin{equation} 
\label{coercivity-Psi} 
\Phi_X(h(\cdot+X)) \geq c \left\|h\right\|_{H^1(\R^3)}^2  .
\end{equation}
\end{lem}
\begin{proof}
Define \begin{align*}
 \mathcal{A} &=   \bigg\{ f \in H^1(\R^3): \;  \re \intr \Delta Q \, f=\im \intr Q f=\re \intr \partial_{x_k}Q \, f =0 , \; k=1,2,3\bigg\},  \\
 \mathcal{B} &=   \mathrm{span} \, \{ iQ,\Delta Q  ,\partial_{x_1}Q,\partial_{x_2}Q,\partial_{x_3}Q  \},
 \end{align*}
then we write $h(\cdot+X)=\tilde{h}(\cdot+X)+r(\cdot+X)$ with $\tilde{h}(\cdot+X) \in \mathcal{A}$ and $r(\cdot+X) \in  \mathcal{B}.$ \\

By \eqref{phi_R^3} and \eqref{Coervicity-R}, we have $$ \Phi(\tilde{h}(\cdot+X)) \geq c \| \tilde{h} \|^2_{H^1(\R^3)}.$$
Since  $r(\cdot+X) \in \mathcal{B},$ we can write $r$ as $$ r(\cdot+X)= \sum_{k=1}^3  \alpha_k \partial_{x_k}Q + \beta i Q + \gamma \Delta Q . $$
Taking the real $L^2$-scalar product in $\R^3$ of $r$ with $iQ$ and using the fact that $Q$ is radial, we get
 \begin{align*} 
 \beta&=\frac{1}{\left\| Q \right\|_{L^2 (\R^3)}^2}(r(\cdot+X),iQ)=\frac{1}{\left\| Q \right\|_{L^2(\R^3)}^2}((h(\cdot+X)-\tilde{h}(\cdot+X), iQ) \\
 & = \frac{1}{\left\| Q \right\|_{L^2 (\R^3)}^2} \left( \im \intr h(x+X) Q(x) \, dx-\im \intr \tilde{h}(x+X) Q(x)\, dx \right).
 \end{align*}
By the definition of $\tilde{h}, $ we have $\im \int \tilde{h}(x+X) Q(x) dx=0.$ Using the orthogonality conditions in Lemma \ref{lem-coercivity-Psi} and the exponential decay of $Q$ from Lemma \ref{compac-supp-argu}, we obtain
\begin{align*}
 \beta&= \frac{1}{\left\| Q \right\|_{L^2 (\R^3)}^2}   \im \intr h(x+X) Q(x) dx\\ &= \frac{1}{\left\| Q \right\|_{L^2}^2}  \im \intr h(x+X) Q(x) \Psi(x+X) \, dx \\&-  \frac{1}{\left\| Q \right\|_{L^2 (\R^3) }^2}  \im \intr h(x+X) Q(x) (\Psi(x+X)-1) dx
\\&  =O(e^{-|X|} \left\| h \right\|_{H^1 (\R^3)}).
\end{align*}

 Similarly, by taking the scalar product of $r$   with $\Delta Q$ and $ \partial_{x_k} Q$ and using the fact that $Q$ is radial, the orthogonality condition in Lemma \ref{lem-coercivity-Psi} and the exponential decay of $Q$ from Lemma \ref{compac-supp-argu},  we obtain $\gamma=\alpha_k=O(e^{-|X|} \left\| h \right\|_{H^1 (\R^3) }).$ \\
 
Thus,  \begin{align*}
  \left\| r \right\|_{H^1  (\R^3) } &\leq C e^{-|X|} \left\| h \right\|_{H^1 (\R^3)} , \\ 
    \left| \Phi_{X}(r(\cdot+X)) \right| &\leq e^{-2|X|} \left\| h \right\|_{H^1 (\R^3)}^2   .
 \end{align*}
 We now have
 \begin{align*}
            \Phi_{X}(h(\cdot+X))=\Phi_{X}(\tilde{h}(\cdot+X))+\Phi_{X}(r(\cdot+X))+2 B_{X}(\tilde{h}(\cdot+X), r (\cdot+X)),
 \end{align*}
where the bilinear form $B_{X}$ is defined as 
\begin{align*} B_{X}(f, g ) &:=  \frac{1}{2} \int  \bigg ( \nabla f_1(x) \nabla g_1(x) + f_1(x) \, g_1(x) - 3Q^2(x) \Psi^2(x+X) f_1(x) \;g_1(x) \bigg )\, dx  \\
&+ \frac 12 \int \bigg( \nabla f_2(x) \nabla g_2(x) +f_2(x) g_2(x) -Q^2(x) \Psi^2(x+X) f_2(x) ) \; g_2(x) \bigg) \, dx .
\end{align*}
 Note that 
 $$\left| B_{X}(\tilde{h}(\cdot+X), r(\cdot+X) )\right| \leq e^{-|X|} \left\| h \right\|_{H^1 (\R^3)}  .$$
 Then, 
 \begin{align*}
  \Phi_{X}(h(\cdot+X))= \Phi(\tilde{h}(\cdot+X)) + O\left(  e^{-   \left|X\right| } \left\| h \right\|_{H^1 (\R^3)} \right)  \geq c \left\| h \right\|_{H^1 (\R^3)}^2.
 \end{align*}
 This implies that there exists $c,R>0$ such that for $|X|>R$
\begin{equation*}
\Phi_{X}(h(\cdot+X))  \geq  c \left\| h \right\|_{H^1 (\R^3)}^2.
\end{equation*}
 \end{proof}

\subsection{Cauchy theory and profile decomposition} 
Next, we review tools needed in Section~\ref{Compactness properties} to prove the compactness property, up to space translation, of a critical solution of the \NNls equation, using a profile decomposition. We use the same notations as in \cite{KiVisnaZhang16}.
Without loss of generality, we assume that $ 0 \in \Theta =\Omega^c$ and $\Theta \subset B(0,1).$        
 We define  $\chi $ to be a smooth cutoff function in $\R^3$
 \begin{align*}
\chi(x)=\left\{\begin{array}{cc} &1 \qquad \left|x\right| \leq \frac{1}{4}, \\ \\
&0 \qquad |x| >  \frac{1}{2}.  \end{array}\right.
 \end{align*}
We define spaces $S^k(I), \; k=0,1,$ as follows \\ 
 \begin{align*}
S^0(I)&=L^\infty_t L^2_x(I \times \Omega) \cap L^{\frac{5}{2}}_t L^{\frac{30}{7}}_x   (I \times \Omega),  \\ 
S^1(I)&=\{ u:I \times \Omega \longrightarrow \C \, | \,  \;  u \mbox{ and } (-\Delta_\Omega)^{\frac{1}{2}}u \in S^0(I) \}.
\end{align*}
\begin{rem}
In order to avoid the endpoints in Strichartz estimates for an exterior domain, see Theorem \ref{Strichartz} below, we take a specific pair $(\frac 52, \frac{30}{7})$, for simplicity.
However, one could use another pair $(p,q)$ with $p=2+\varepsilon$ and $q=\frac{6(2+\varepsilon)}{2+3 \varepsilon}$ instead of $(\frac 52, \frac{30}{7}),$ where $\varepsilon>0$ is small enough. 
\end{rem}
By interpolation, 
$$ \left\| u \right\|_{L^{q}_t L^r_{x}(I \times \Omega)} \leq \left\| u \right\|_{S^0(I)} , \quad \text{ for all } \frac 2q +\frac 3r=\frac 32 \text{ with } \frac 52 \leq q \leq \infty. $$
Similar estimates hold for $S^1(I).$ We will, in particular, use $(q,r)$ equal to $(5,\frac{30}{11})$ and $(\infty,2).$ \\ 

One particular Strichartz space we use is 

$$ X^1(I):=L^5_t H^{1,\frac{30}{11}}_0(I\times \Omega) .   $$

Note that, $S^1(I) \subset X^1(I)$ and by Sobolev embedding, there exists $C>0$ such that $  \left\| f \right\|_{L^5_{t,x}   \left(I \times \Omega \right) } \leq C \left\| f \right\|_{X^1(I)}.$  \\ 

We next define $N^0(I)$ as the corresponding dual of $S^0(I)$ and 
\begin{equation}  
\label{def-N1}
N^1(I)=\{u:I \times \Omega \longrightarrow \C \,  |\,  \; u \mbox{ and }(-\Delta_\Omega)^{\frac{1}{2}}u \in N^0(I) \} . 
\end{equation}
Then, we have
\begin{equation}
\label{N0-estim}
\left\| u \right\|_{N^0(I)} \leq \left\| u \right\|_{L^{q^{'}}_t L^{ r^{'}  }_x (I \times \Omega )}   \qquad  \text{ for all }  \quad \frac 2q +\frac 3r=\frac 32 \text{ with } \; \frac 52 \leq q \leq \infty,    
\end{equation}
\begin{equation*}
\text{where}  \quad \frac 1q+\frac{1}{q^{'}}=1 \quad  \text{ and } \quad  \frac 1r+\frac{1}{r^{'}}=1 .
\end{equation*}

In particular, we will use $(q{'},r^{'})=(\frac 53, \frac{30}{23}),$ the H\"older dual to the Strichartz pair $(q,r)=(\frac 52, \frac{30}{7}). $ One can get a similar estimate to \eqref{N0-estim} for $N^1(I)$ using the same pair, see Theorem~\ref{Strichartz}. \\

Next, we state the Strichartz estimates using the above pairs and other necessary results from \cite{KiVisnaZhang16}.
\begin{theo}[Strichartz estimates, \cite{Ivanovici10}]
\label{Strichartz}
Let $I$ be a time interval and $t_0 \in I.$ Let $u_0 \in H^1_0(\Omega),$ then there exists a constant $C>0$ such that the solution $u(t,x)$ to the nonlinear Schr\"odinger equation on $ \R \times \Omega$ with Dirichlet boundary condition
 \begin{align*}
\begin{cases}
 i \partial_{t} u + \Delta_{\Omega} u = f \; \text{ on } \R \times \Omega\\
 u(0,x)=u_0(x)  \\
 u_{ \mid \partial \Omega}=0
 \end{cases}
\end{align*}
satisfies 
\begin{equation*}
\left\| u \right\|_{ S^{0}(I) } \leq C \left(  \left\| u_0 \right\|_{L^2(\Omega)}   + \left\| f \right\|_{N^0(I)} \right), 
\end{equation*}
and 
\begin{equation}
\label{strich-S1}
\left\| u \right\|_{S^1(I)}       \leq C \left(  \left\| u_0 \right\|_{H^1_0(\Omega)}   + \left\| f \right\|_{N^1(I)} \right) .
\end{equation}
In particular, 
$$ \left\| u \right\|_{ X^1 (I \times \Omega)}      \leq C \left(   \left\| u_0 \right\|_{H^1_0 ( \Omega ) }    +
 \left\| f \right\|_{  L^{\frac{5}{3}} H^{1 ,\frac{30}{23}}_{0}  (I \times \Omega ) } \right) . $$
\end{theo}

\begin{prop}[Local smoothing, {\cite[Corollary 2.14]{KiVisZha16}} ]
\label{local smoothing}
Given $\omega_0 \in H^1_0(\Omega)$, we have \\ 
\begin{equation*}
 \left\| \nabla e^{it\Delta_\Omega} \omega_0 \right\|_{L^{\frac{5}{2}}L^{\frac{30}{17}} ( |t-\tau| \leq T , \, |x-z|\leq R )} \leq R^{\frac{31}{60}} T^{\frac{1}{5}} \left\| e^{it \Delta_\Omega} \omega_0\right\|_{L^5_{t,x}(R\times \Omega)}^{\frac{1}{6}} \left\| \omega_0 \right\|_{H^1_0(\Omega)}^{\frac{5}{6}},
 \end{equation*}
uniformly in $\omega_0$ and the parameters $R,T>0$, $z\in \R^3$ and $\tau \in \R$.
\end{prop}

\begin{lem}[Stability,\cite{KiVisnaZhang16}]
\label{stability}
Let $I\subset \R$ be a time interval and let $\tilde{u}$ be an approximate solution to (NLS$_\Omega$) on $ I \times \Omega$ in the sense that $$i \partial_t \tilde{u} + \Delta_{\Omega} \tilde{u}= - \left| \tilde{u} \right|^2 \tilde{u} + e    \; \text{ for some function } e. $$
Assume that $$  \left\| \tilde{u} \right\|_{L^{\infty} H^1_0(I\times \Omega)} \leq \mathcal{E} \; \text{ and } \; \left\| \tilde{u} \right\|_{L^5_{t,x}(I \times \Omega)} \leq L    $$
for some positive constants $ \mathcal{E}$ and $L$. Let $t_0 \in I $ and $u_0 \in H^1_0(\Omega)$ and assume the smallness conditions $$ \left\| \tilde{u}(t_0) - u_0 \right\|_{H^1_0(\Omega)} \leq \varepsilon \; \text{ and } \; \left\| e \right\|_{N^1(I)} \leq \varepsilon$$
for some $0<\varepsilon < \varepsilon_1=\varepsilon_1(\mathcal{E},L).$ Then there exists a unique solution $u:I \times \Omega \longrightarrow \C$ to (NLS$_\Omega$) with initial data $u(t_0)=u_0$ satisfying  $$ \left\| u - \tilde{u} \right\|_{X^1(I\times \Omega)} \leq C( \mathcal{E}, L ) \varepsilon.   $$
\end{lem}
\begin{theo}[Linear profile decomposition in $H^1_0(\Omega)$, {\cite[Theorem 3.2]{KiVisnaZhang16}}]
\label{profile decom}
Let $\{f_n\} $ be a bounded sequence in $H^1_0(\Omega).$ After passing to a subsequence, there exist $J ^*\in \{ 0,1,2,...., \infty\},\\ \{ \phi_n^j \}_{j=1}^{J^*} \subset H^1_0(\Omega)\setminus \{0\},$ $\{t_n^j\}_{j=1}^{J^*} \subset \R  $ such that, for each $j$ either $t_n^j \equiv 0$ or $t_n^j \rightarrow \pm \infty$ and $\{x_n^j  \}_{j=1}^{J^*}\subset \Omega$ conforming to one of the following two cases for each $j:$ \\
Case 1: $x_n^j =0$ 
and there exists $\phi^j \in H^1_0(\Omega) $ so that $\phi_n^j:= e^{i t_n^j \Delta_{\Omega} } \phi^j. $ \\ 
Case 2: $|x_n^j|   \rightarrow \infty$   
 and there exists $\phi^j \in H^1(\R^3) $ so that $$ \phi_n^j:= e^{it_n^j \Delta_{\Omega}}[(\chi_n^j \phi^j)(x-x_n^j)] \quad 
\text{with} \quad \chi_n^j(x):= \chi\left(\frac{x}{|x_n^j|}\right) .$$
Moreover, for any finite $0\leq J \leq J^*$ we have the decomposition $$ f_n= \sum_{j=1}^{J} \phi_n^j + \omega_n^J  $$ with the remainder $\omega_n^J \in H^1_0(\Omega) $ satisfying 
\begin{align}
\label{limJ limsup_n omega}
    \lim_{J \rightarrow J^*} \limsup_{ n \rightarrow \infty} \left\| e^{it \Delta_{\Omega}} \omega_{n}^J  \right\|_{L^5_{t,x}(\R \times \Omega)} &=0, \\ 
\forall J \geq 1, \quad    \lim_{n \rightarrow \infty } \bigg\{ M[f_n]- \sum_{j=1}^{J} M[\phi_n^j]-M[\omega_n^J]\bigg\}  &=0,  \\
\forall J \geq 1, \quad \quad  \lim_{n \rightarrow \infty} \bigg\{ E[f_n]- \sum_{j=1}^{J} E[\phi_n^j]-E[\omega_n^J] \bigg\}&=0, \\ 
 \label{asymptotic ortho para}
     \lim_{n \rightarrow \infty } \left| x_n^j - x_n^k \right| + \left|t_n^j -t_n^k  \right|=& \infty \;  \text{ for each } j \neq k . 
 \end{align}
\end{theo}
\begin{theo}[{\cite[Theorem 4.1]{KiVisnaZhang16}}]
\label{contruction of v_n if d(x_n)}
Let $\{t_n\} \subset \R $ be such that $t_n \equiv 0$ or $t_n \rightarrow \pm \infty$. Let $\{x_n\} \subset \Omega$ be such that $|x_n|$ tends to $\infty,$ as $n$ goes to $\infty.$  Assume $\phi \in H^1(\R^3)$ satisfies 
\begin{align}
     \left\| \nabla \phi \right\|_{L^2(\R^3)}  \left\| \phi \right\|_{L^2(\R^3)} &< \left\| \nabla Q \right\|_{L^2(\R^3)}  \left\| Q \right\|_{L^2(\R^3)}  ,\\
     \label{energy-mass-phi-Q}
 M_{\R^3}[\phi] E_{\R^3}[\phi] &< M_{\R^3}[Q] E_{\R^3}[Q]. 
\end{align}
Define $$ \phi_n:=e^{it_n \Delta_\Omega}\left[ (\chi_n \phi)(x-x_n) \right] \quad \text{with} \quad \chi_n(x):= \chi\left(\frac{x}{|x_n|} \right) .$$
Then, for $n $ sufficiently large, there exists a global solution $v_n$ to \Nls with initial data $v_n(0):=\phi_n$, which satisfies $$ \left\| v_n \right\|_{L^5_{t,x}(\R \times \Omega)} \leq C (\left\| \phi \right\|_{H^1(\R^3)} ) . $$
Furthermore, for any $\varepsilon > 0 $ there exists $N_\varepsilon \in \N$ and $\psi_\varepsilon \in C_c( \R \times \R^3)$ such that, for all $n \geq N_\varepsilon,$
\begin{equation}
\label{v_n^j-psi-R3}
\left\| \underline{v}_n(t-t_n, x+x_n) - \psi_{\varepsilon}(t,x) \right\|_{L^5 H^{1,\frac{30}{11}}(\R\times \R^3)} < \varepsilon .
\end{equation}
\end{theo}
\begin{rem}                                                                                                                                                                                                            
Note that, we have made a slight modification in the notation of the above Theorem \ref{contruction of v_n if d(x_n)}, in order to keep the consistent notation in this paper. We denote $\underline{v}_n$ the extension of the solution $v_n$ by $0$ on $\Omega^c$ such that $\underline{v}_n \in H^1(\R^3).$
 Let us mention that  $\phi_n$ is well defined in $H^1_0(\Omega).$ Indeed, by the definition of $\chi_n$ and as $|x_n| \rightarrow \infty,$ we have 
 \begin{equation*}
  x \in \partial \Omega \quad \Longrightarrow \quad \chi_n (x) = 0 \quad \text{ as  } n \rightarrow + \infty . 
 \end{equation*}
 
Moreover, one can check that the energy-mass assumption \eqref{energy-mass-phi-Q} is equivalent to the one given in \cite[Theorem 4.1]{KiVisnaZhang16} using the following identity: 
\begin{multline*}
\bigg\{ u_0 \in H^1(\R^3): E_{\R^3}[u_0]M_{\R^3}[u_0] < E_{\R^3}[Q]M_{\R^3}[Q] \bigg\} \\ =\bigcup_{0 < \lambda <\infty} \bigg\{ u_0 \in H^1(\R^3):  E_{\R^3}[u_0]+\lambda M_{\R^3}[u_0] <2 \sqrt{\lambda E_{\R^3}[Q]M_{\R^3}[Q]} \bigg\},
\end{multline*}
which follows by computing the minimum, of $\lambda \mapsto
E_{\R^3}[u_0]+\lambda M_{\R^3}[u_0] -2 \sqrt{\lambda
E_{\R^3}[Q]M_{\R^3}[Q]}$.
\end{rem}
\section{Modulation}
\label{Sec-Modulation}
Let $u \in H^1_0(\Omega)$ and  define $$ \delta(u)= \left| \int_{\R^3} \left|  \nabla Q  \right|^2  -  \int_{\Omega} \left|  \nabla u \right|^2  \,  \right| .$$ 
Assume that
\begin{equation}
\label{MEu=MEQ}
M[u]=M_{\R^3}[Q] \quad \text{and} \quad  E[u]=E_{\R^3}[Q] .
\end{equation}
\begin{lem}
Let $ u \in H^1_0(\Omega)$ satisfying \eqref{MEu=MEQ} and $\delta(u)$ small enough. Then there exists $X_0 \in \R^3$ large and $\theta_0 \in \R$ such that 
\begin{equation}
 e^{-i \theta_0}  u(x)= Q(x-X_0) \Psi(x)+ h(x)
\end{equation}
with $\lVert h \rVert_{H^1_0(\Omega)} \leq   \tilde{ \varepsilon}( \delta( u )),$ where
 $\tilde{ \varepsilon}( \delta( u )) \rightarrow 0$ as  $\delta( u ) \rightarrow 0 .$
\end{lem}
\begin{proof}
Let  $\uu \in H^1(\R^3)$ be defined as above in \eqref{defunderlineU}
and observe that $\delta(u)=\delta(\uu).$
By Proposition \ref{uConvQ}, since
 \begin{equation}
\label{MEubar=MEQ}
M[u]=M_{\R^3}[\uu]=M_{\R^3}[Q], \quad E[u]=E_{\R^3}[\uu]=E_{\R^3}[Q] ,
\end{equation}
and $\delta(\uu)$ being small enough,  there exist $ \theta_0 \in \R$ and $X_0 \in \R^3$ such that $$e^{-i \theta_0 } \uu(x)=Q(x-X_0)+ \tilde{h}(x)$$ 
with 
$ \Vert \tilde{h} \rVert_{H^1(\R^3)} \leq \tilde{ \varepsilon}( \delta( \uu )),$
 where $\tilde{\varepsilon}(\delta(\uu)) \longrightarrow 0 $ as $\delta(\uu) \longrightarrow 0.$   \\ 

Moreover, if $x \in \Omega^c,$ then $\uu(x) =0,$ which implies that 
\begin{equation}
\label{Q+r=0}
 x \in \Omega^c \quad \Longrightarrow  \quad  Q(x-X_0)+\tilde{h}(x)=0,  
\end{equation}
and for $\delta(\uu)$ small enough, by \eqref{x0large}, $|X_0|$ is large such that $$\frac{e^{-|X_0|}}{|X_0|}\leq C \, \tilde{\varepsilon}(\delta(\uu)) .$$ 

We write, \begin{align*} e^{-i \theta_0} \uu (x)&= Q(x-X_0) \Psi(x)+ (1-\Psi(x))Q(x-X_0) ) + \tilde{{h}} (x)\\
&= Q(x-X_0) \Psi(x)+ {h}(x).
\end{align*}
Using the fact that $(1-\Psi)$ has a compact support, $Q$ having an exponential decay, $|X_0|$ being large, and Lemma \ref{compac-supp-argu}, we get  $$\lVert h \rVert_{H^1(\R^3)} \leq  \tilde{ \varepsilon}( \delta( \uu ))+  C \frac{e^{-|X_0|}}{|X_0|} \leq \tilde{ \varepsilon}( \delta( \uu )).$$ \\ 
By \eqref{Q+r=0} and the definition of $\Psi $ in \eqref{def-psi}, we have
$$h(x)=0,\quad  \text { if  } \; x \in \Omega^c.$$
Thus,  $h(x)=0$ on $\partial \Omega$ and $h(x)\in H^1_0(\Omega),$ which concludes the proof.
\end{proof}

\begin{lem}
\label{modulation}
There exists $\delta_0>0$ and a positive function $\varepsilon (\delta)$ defined for $0<\delta \leq \delta_0,$  which tends to $0$ when $\delta \rightarrow 0,$ such that for any ${u} \in H^1_0(\Omega)$ satisfying \eqref{MEu=MEQ} and $\delta(u)<\delta_0,$ there exists a couple $(\mu,X)\in \R \times \R^3$ such that the following holds 
 \begin{equation}
 \left\|u(x) - Q(x-X) \Psi(x) e^{i\mu} \right\|_{H^1_0(\Omega)} \leq \varepsilon(\delta),
\end{equation}
\begin{align}
  \re \int_{\Omega}  u(x) \;  \partial_{x_k} (Q(x-X) \Psi(x)) e^{-i\mu} \,  dx &=0, \; \quad k=1,2,3,  \\
\im  \int_{\Omega}  u(x)\;  Q(x-X) \Psi (x) e^{-i\mu}  dx &=0. 
 \end{align}
   The parameters $\mu$ and $X$ are unique in $\R \slash \pi \Z  \times \R^3$ and the mapping $u\rightarrow (\mu,X)$ is $C^1.$
\end{lem}
\begin{proof}
Let \begin{align*}
\Phi: H^1_0(\Omega) \times \R^3 \times \R & \longrightarrow \R^4 \\ 
        (u\;,\;X\;,\;\mu) & \longmapsto  \left(\Phi_k(u,X,\mu)\right)_{1 \leq k \leq 4} ,
\end{align*}
where \begin{align*}
\Phi_k(u,X,\mu)&:=\re \intom u(x)  \; \partial_{x_k} (Q(x-X)\Psi(x)) e^{-i \mu}  \, dx, \; k=1,2,3, \\
\Phi_4(u,X,\mu)&:= \im \intom u(x) \;  Q(x-X) \Psi(x)\,  e^{-i\mu} dx.
\end{align*}
Let $X_0 \in \R^3.$ Note that $\Phi(Q(\cdot-X_0 ) \Psi,X_0,0)=0.$ Indeed, integrating by parts, we get 
 \begin{align*}
 \Phi_k( Q(\cdot-X_0 ) \Psi,X_0,0)&=\re \intom Q(x-X_0)\Psi(x) \partial_{x_k} (Q(x-X_0)\Psi(x)) \, dx \\ 
 &= \frac 12 \re \intom  \partial_{x_k} (    (Q(x-X_0)\Psi(x))^2  ) \, dx =0, \\ 
 \Phi_4( Q(\cdot-X_0 ) \Psi,X_0,0)&=\im \intom  Q(x-X_0)^2\Psi(x)^2 \, dx=0.
 \end{align*}
 
 \begin{itemize}
\item \underline{\bf Step 1:} Computation of $d_{ \left( X , \mu \right)} \Phi_k$. \\ 

We have 
\begin{align*}
\frac{\partial}{\partial X_j} \Phi_k(u,X,\mu)&= - \re \intom e^{-i \mu} u(x) \partial_{x_k}( \partial_{x_j} Q(x-X) \Psi(x)) \, dx .
\end{align*}
Integrating by parts, we obtain
\begin{align*}
\frac{\partial}{\partial X_j} \Phi_k( Q(\cdot-X_0 ) \Psi,X_0,0)&= \re \intom \partial_{x_j} Q(x-X_0) \Psi(x) \partial_{x_k}( Q(x-X_0)\Psi(x)) \, dx. 
\end{align*}

If $k=j,$ we have 
\begin{align*}
\frac{\partial}{\partial X_j} \Phi_k( Q(\cdot-X_0 ) \Psi,X_0,0)&=\re \intom (\partial_{x_j} Q(x-X_0))^2 \, dx  \\&+ \re \intom  ( \partial_{x_j} Q(x-X_0))^2 (\Psi(x)^2-1)\, dx  \\& + \re \intom Q(x-X_0) \partial_{x_j}Q(x-X_0) \Psi(x) \partial_{x_j}\Psi(x) \, dx. 
\end{align*}
Since $\partial_{x_j} \Psi$ and $(\Psi^2 -1)$ have a compact support and $Q$ has an exponential decay, we deduce
\begin{align*}\frac{\partial}{\partial X_j} \Phi_k( Q(\cdot-X_0 ) \Psi,X_0,0)&= \left\| \partial_{x_j} Q \right\|_{L^2(\R^3)}^2+ O(e^{-2|X_0|}) \\
&=\frac 13 \left\| \nabla Q \right\|_{L^2(\R^3)}^2+ O(e^{-2|X_0|}).
\end{align*}
If $k\neq j,$ then 
\begin{align*}
\frac{\partial}{\partial X_j} \Phi_k( Q(\cdot-X_0 ) \Psi,X_0,0)&= \re \intom \partial_{x_j} Q(x-X_0) \Psi(x) \partial_{x_k}( Q(x-X_0)\Psi(x)) \, dx \\
&=\re \intom \partial_{x_j} Q(x-X_0) \partial_{x_k} Q(x-X_0) dx \\&+ \re \intom \partial_{x_j} Q(x-X_0) \partial_{x_k} Q(x-X_0) (\Psi(x)^2 - 1) dx
\\ &+ \re \intom \partial_{x_j}Q(x-X_0)  \Psi(x)   Q(x-X_0) \partial_{x_k} \Psi(x) dx. 
\end{align*}
Using the same argument as before and the fact that $Q$ is radial $( \int  \partial_{x_j} Q \partial_{x_k} Q =0 ,$ if $k \neq j$), we obtain 
$$ \frac{\partial}{\partial X_j} \Phi_k (Q(\cdot-X_0 ) \Psi,X_0,0)= O(e^{-2 |X_0|}). $$
Next, we compute $ \frac{\partial}{\partial {\mu}}  \Phi_k(u,X,\mu)$:  
\begin{align*}
\frac{\partial}{\partial{\mu}}  \Phi_k(u,X,\mu)&=  \re \intom -i e^{-i \mu } u(x) \partial_{x_k}(Q(x-X) \Psi(x) )dx, \\
\frac{\partial}{\partial{\mu}}  \Phi_k ( Q(\cdot-X_0 ) \Psi,X_0,0)&=  \im \intom Q(x-X_0) \Psi(x) \partial_{x_k}(Q(x-X_0)\Psi(x) )\, dx=0 . \\
\end{align*}

\item \underline{\bf Step 2:} Computation of $d_{(X,\mu)} \Phi_4.$ \\ 

We have 
\begin{align*}
 \frac{\partial}{\partial {X_j} }\Phi_4( u,X,\mu)&=- \im \intom e^{-i\mu}  u(x)  (\partial_{x_j}Q(x-X)\Psi(x)) \, dx ,
 \end{align*}
and thus, 
 \begin{align*}
 \frac{\partial}{\partial {X_j}} \Phi_4 (Q(\cdot-X_0 ) \Psi,X_0,0)&= - \im \intom Q(x-X_0)\Psi(x)  \partial_{x_j} ( Q(x -X_0) \Psi(x)) \, dx =0 . 
\end{align*}
Also,
\begin{align*}
\frac{\partial}{\partial \mu }\Phi_4( u,X,\mu)&= \im \intom - i e^{-i \mu} u(x) Q(x-X) \Psi(x) \, dx , \\ 
 \frac{\partial}{\partial {\mu}} \Phi_4 ( Q(\cdot-X_0 ) \Psi,X_0,0)&= -  \intom  Q(x-X_0)^2 \Psi(x) ^2  = -\left\| Q \right\|_{L^2(\R^3)} ^2+ O(e^{-2|X_0|}) . 
 \end{align*}
 \item \underline{\bf Step 3:} Conclusion.  \\ 
 
Combining Step 1 and Step 2, we get
\begin{align*}
d_{(X,\mu)} \Phi( Q(\cdot-X_0 ) \Psi,X_0,0)&= \begin{pmatrix} 
\frac 13 \left\| \nabla  Q  \right\|_{L^2(\R^3)}^2 & 0 & 0 & 0 \\ 
0 & \frac 13  \left\| \nabla Q \right\|_{L^2(\R^3)}^2 & 0 & 0 \\
0 & 0 & \frac 13 \left\| \nabla Q \right\|_{L^2(\R^3)}^2 & 0 \\
0 & 0 & 0 & - \left\| Q  \right\|_{L^2(\R^3)}^2
\end{pmatrix} \\
&+ O(e^{-2|X_0|}). 
\end{align*}
We can deduce that $d_{(X,\mu)} \Phi$ is invertible at $(Q(\cdot-X_0 ) \Psi( \cdot ),X_0,0),$ if $|X_0|$ is large. Then, by the implicit 
function theorem there exists $\epsilon_0, \eta_0>0$ such that for $u\in H^1_0(\Omega),$ we have 
$$ \left\| u(\cdot)-Q(\cdot-X_0) \Psi( \cdot) \right\|_{{H}^1_0(\Omega)}^2 < \epsilon_0 \Longrightarrow \exists ! (X,\mu): \quad |\mu|+|X-X_0| \leq \eta_0 \; \text{ and } \;   \Phi(u,X,\mu)=0.   $$
\end{itemize}
\end{proof}

Let $u(t)$ be a solution of \Nls satisfying \eqref{MEu=MEQ}. In the sequel we write $$ \delta(t)=\delta(u(t)).$$
Let $D_{\delta_0}=\{ t \in I : \;  \delta(t) < \delta_0\}, $ where $I$ is the maximal time interval of existence of $u.$ \\ 
By Lemma \ref{modulation}, we can define $C^1$ functions $X(t)$ and $\mu(t)$ for $t \in D_{\delta_0}.$

We now work with the parameter $\theta(t)= \mu(t)-t.$ Write 
\begin{equation}
\label{decom-u}
e^{-i\theta(t)-it} u(t, x)=(1+\rho (t))Q(x-X(t)) \Psi(x)+h(t, x), 
\end{equation}
where $h(x) \in H^1_0(\Omega)$ and define
$$ \rho(t)=\displaystyle \re  \frac{ e^{-i \theta(t)-it}  \int \nabla  \bigg(Q(x-X(t)) \Psi(x) \bigg) \cdot \nabla \uu(t,x) dx   }{\int \left| \nabla  \big(Q(x-X(t)) \Psi(x) \big) \right| ^2\, dx }-1.$$

This implies that \\ 
\begin{equation} 
\label{decom-u-trans}
e^{-i\theta(t)-it} \uu(t, x+X(t))=(1+\rho (t))Q(x) \Psi(x+X(t))+\bh(t, x+X(t)), 
\end{equation}
where $\bh(x) \in H^1(\R^3)$ is defined by
 \begin{equation*}
 \bh(t,x)=\begin{cases}
 h(t,x) \quad &\forall x \in \Omega, \\
 0 \qquad  &\forall x \in \Omega^c.
 \end{cases}
\end{equation*}

One can see that $\rho(t)$ is chosen such that $h$ satisfies the orthogonality condition 
\begin{multline} 
\label{rho-ortho-DQh} 
 \re \intom \Delta(Q(x-X(t))\Psi(x)) h(t,x) \, dx \\  = \re \int \Delta(Q(x)\Psi(x+X(t))) \bh(t,x+X((t)) \, dx =0 .
\end{multline}

By Lemma \ref{modulation}, $h$ also satisfies the orthogonality conditions 
\begin{equation}
\label{modu-ortho-dQh-im}
\im \intom h(t,x) Q(x-X(t)) \Psi(x) \, dx = \im \int \bh(t,x+X(t)) Q(x) \Psi(x+X(t)) \, dx =0,
\end{equation}
and 
\begin{multline}
\label{modu-ortho-dQh}
\re \intom h(t, x)     \partial_{x_k}    ( Q(x-X(t)) \Psi(x) ) \,dx \\ =\re \int \bh(t, x+X(t))     \partial_{x_k}    ( Q(x) \Psi(x+X(t)) ) \,dx =0, \quad k=1,2,3.
\end{multline}

In the following lemma, to simplify notation, we denote $f(\cdot+X)$ by $f_{_X}(\cdot)$ for any function $f.$ If $f$ is a complex function, then we denote by $f_{1_X}(\cdot)$ the real part of $f_{_X}$ and by $f_{2_X}(\cdot)$ the imaginary part. 

\begin{prop}
\label{lem-equiv-modu-param}
Let $u(t)$ be a solution of \Nls satisfying \eqref{MEu=MEQ}. Then the following estimates hold for $ t \in D_{\delta_{0}}$
\begin{multline}
\label{equiv}
 \left| \rho(t) \right| +O \left( \ee \right) \approx  \left| \int Q\Psi_{_X} \thre dx \right|  +O \left( \ee \right) \approx  \delta(t) +O \left( \ee \right)  \\ \approx \left\| h(t) \right\|_{H^1_0(\Omega)} +O \left( \e \right) .
\end{multline}
\end{prop}
\begin{proof}
Let  $\dt(t)= | \rho(t) | + \left\| \bh \right\|_{H^1} + \delta(t)$, which is small, if $\delta(t)$ is small. 
By the expansion of $u$ in \eqref{decom-u-trans} we have $e^{-i\theta(t)-it} \uu(t, x+X(t))=(1+\rho (t))Q(x)\Psi_{_X}(x)+\th(t, x),$ thus, if $x+X(t) \in \Omega,$ then $ \uu(t,x+X(t))=u(t,x+X(t)),$ otherwise $\uu(t,x+X(t))=0.$  \\
 \begin{itemize}
\item \underline{\bf Step 1:} Approximation of $|\rho|$ using the mass conservation. \\  

Since $M[u]=M_{\R^3}[\uu]=M_{\R^3}[Q\Psi_{_X}+ \rho Q\Psi_{_X} + \th]=M_{\R^3}[Q]$, we have,
\begin{equation}
\label{M(Q+g)}
\int \bigg( Q^2( \Psi^2_{X}-1)+ 2 \rho\,  Q^2 \Psi^2_{X}+  2 \rho \,Q\Psi_{_X} \thre + \rho^2 Q^2 \Psi^2_{X} + 2  Q\Psi_{_X} \thre + |\th|^2 \bigg) dx= 0 .
\end{equation}

Using \eqref{M(Q+g)} and Lemma \ref{compac-supp-argu}, we obtain  
\begin{align*}
2 | \rho | \left|  \int  \; Q^2 \, \Psi^2_{_X} \right| &= \bigg| 2 \int Q \Psi_{_X} \thre + \int Q^2 ( \Psi_{_X}^2 -1 ) + 2 \rho \int  Q \Psi_{_X} \thre  + \rho^2 \int Q^2 \Psi_{_X}^2 + \int | \th |^2   \bigg| \\
&=   2  \left| \int Q\Psi_{_X} \thre \, \right|    + O\left(  \dt^2+ \ee \right),
\end{align*}

which yields
 \begin{equation}
\label{rho+Qh1}
| \rho | = \frac{1}{M[Q]}    \left| \int Q\Psi_{_X} \thre \, dx \right|    + O\left(  \dt^2+ \ee \right).
\end{equation}
\item \underline{\bf Step 2:} Approximation of $|\rho|$ in terms of $\delta$.\\ 

By the definition of $\delta(t)$, we have
\begin{align*}
\delta(t)&= \left| \int \left| \nabla(Q\Psi_{_X}+\rho Q\Psi_{_X}+\th) \right|^2 dx - \int \left| \nabla Q \right|^2 dx \right|  \\
            &= \bigg| \int \left| \nabla(Q\Psi_{_X}) \right|^2 + 2 \rho \left| \nabla (Q\Psi_{_X} ) \right|^2 + \rho^2 \left| \nabla(Q\Psi_{_X}) \right|^2 + 2 \rho \nabla (Q\Psi_{_X}) \cdot \nabla \thre  \\ &+ 2 \nabla (Q\Psi_{_X}) \cdot \nabla \thre + \left| \nabla \th \right|^2 -\int \left| \nabla Q \right|^2  \bigg| .
 \end{align*}
 Using integration by parts and the orthogonality condition \eqref{rho-ortho-DQh}, we get
 \begin{multline*}
         \delta(t)   = \bigg| \int \left| \nabla Q \right|^2 (\Psi_{_X}^2-1)  + 2\,   \nabla Q \cdot \nabla\Psi_{_X}   Q\Psi_{_X} +  Q^2 \left| \nabla\Psi_{_X} \right|^2   \\
          + (2\rho +\rho^2) \int  \left| \nabla(Q\Psi_{_X}) \right|^2    +  \int   \left| \nabla \th \right| ^2  \bigg|             .
\end{multline*}
Using the fact that $(\Psi^2 -1)$ and  $\nabla\Psi$ have compact supports and applying Lemma \ref{compac-supp-argu}, we get
  \begin{equation}
\label{rho+delta}
\left| \rho \right|= \frac{\delta}{2 \left\| \nabla Q \right\|_{L^2(\R^3)}^2}  + O\left(\dt^2 + \ee \right).
\end{equation}
\item \underline{\bf Step 3:} Energy and Mass conservation. \\

We define: $g=\rho Q\Psi_{_X} + \th  .$ Since $ E_{\R^3}[\uu]=E_{\R^3}[Q\Psi_{_X}+g]=E_{\R^3}[Q],$ we have
\begin{align}
\label{A_0+A_L}
    \frac 12 \int \left| \nabla(Q\Psi_{_X}) \right|^2 &- \frac 12 \int \left| \nabla Q \right|^2 - \frac 14 \int Q^4\Psi_{_X}^4 + \frac 14 \int Q^4+ \int \nabla(Q\Psi_{_X}) \cdot \nabla g_1 - \int Q^3\Psi_{_X}^3 \,  g_1 \\
\label{ANL}
     &+ \frac 12 \int \left| \nabla g \right|^2  - \frac 12 \int Q^2\Psi_{_X}^2(3 g_1^2+g_2^2) - \int Q\Psi_{_X} |g|^2g_1 - \frac 14 |g|^4 =0 .
\end{align}
First, we estimate \eqref{A_0+A_L}. For that we denote
\begin{align*}
A_0&=   \frac 12 \int \left| \nabla(Q\Psi_{_X}) \right|^2 - \frac 12 \int \left| \nabla Q \right|^2 - \frac 14 \int Q^4\Psi_{_X}^4 + \frac 14 \int Q^4,\\
A_L(g)&=  \int \nabla(Q\Psi_{_X}) \cdot \nabla g_1 - \int Q^3\Psi_{_X}^3 g_1. \\
\end{align*}
In this step, we estimate $A_0$ and $A_L(g)$. Using the fact that $\nabla\Psi, (\Psi^2-1)$ and $(\Psi^4-1)$ have compact supports and Lemma~\ref{compac-supp-argu}, we have
\begin{equation}
\label{A_0-esti}
A_0 =O\left( \ee \right). 
\end{equation}
Next, we show that 
\begin{multline}
\label{A_L-estim}
A_L(g)= \frac 12 \int |g|^2 -  2 \int \nabla Q. \nabla\Psi_{_X}  \, g_1 - \int Q \Delta\Psi_{_X}  \, g_1  - \int Q^3\Psi_{_X} (\Psi_{_X}^2-1) g_1  \\  + O\left(\ee\right). 
\end{multline}
Integrating by parts, we obtain
\begin{align*}
   \int \nabla(Q\Psi_{_X}) \cdot \nabla g_1 &= - \int \Delta(Q\Psi_{_X}) g_1 = -\int \Delta Q \,\Psi_{_X} g_1 - 2 \int \nabla Q \cdot \nabla\Psi_{_X} g_1 - \int Q \Delta\Psi_{_X} g_1, \\
   - \int Q^3\Psi_{_X}^3 g_1&= -\int Q^3\Psi_{_X}  \, g_1 - \int Q^3\Psi_{_X}(\Psi_{_X}^2-1) g_1.
\end{align*}
Using the equation \eqref{equQ} for $Q,$ we deduce
\begin{align*}
A_{L}(g)=  - \int Q\Psi_{_X}  g_1 - 2 \int \nabla Q \cdot \nabla\Psi_{_X} g_1 - \int Q \Delta\Psi_{_X}  \, g_1 - \int Q^3\Psi_{_X}(\Psi_{_X}^2-1) g_1.
\end{align*}

Since $M[u]=M[\uu]=M[Q\Psi_{_X}+g]=M[Q],$ we have
\begin{align*}
    \int Q^2 (\Psi_{_X}^2-1) + 2 \int Q\Psi_{_X}  g_1 + \int |g|^2 =0,    \\ 
    -\int Q\Psi_{_X} g_1= \frac 12 \int |g|^2 + O\left( \ee \right).
\end{align*}
This implies \eqref{A_L-estim}.

\item \underline{\bf Step 4:} Approximation of $\left\| h \right\|_{H^1_0(\Omega)}.$ \\

Recall that $g=\rho\,  Q\Psi_{_X}+ \th .$ In this step we prove 
$$    \left\|h \right\|_{H^1_0(\Omega)}=  O\left( |\rho|+ \dt^{\frac 32}+ \e  \right).$$

Summing up all terms  \eqref{ANL}, \eqref{A_0-esti} and \eqref{A_L-estim}, we obtain 
\begin{align*}
&\frac 12 \int |\rho \, Q\Psi_{_X} + \th |^2 - 2 \int \nabla Q \cdot \nabla\Psi_{_X} (\rho \, Q\Psi_{_X}+\thre) - \int Q \Delta\Psi_{_X} (\rho\, Q\Psi_{_X}+\thre)\\ 
&   - \int Q^3\Psi_{_X}(\Psi_{_X}^2-1) ( \rho Q\Psi_{_X} + \thre) +  \frac 12 \int \left| \nabla ( \rho\, Q\Psi_{_X}+\th ) \right|^2 - \frac 12 \int Q^2\Psi_{_X}^2 ( 3(\rho \, Q\Psi_{_X}+\thre)^2+\thima^2) 
\\ &- \int Q\Psi_{_X} | \rho\, Q\Psi_{_X}+\th |^2( \rho \, Q\Psi_{_X}+\thre) - \frac 14 \int | \rho \, Q\Psi_{_X}+\th|^4  =O\left(\ee\right).
\end{align*}
Denote 
\begin{align*}
B_L(\bh)&=- 2 \int \nabla Q .\nabla\Psi_{_X} (\rho \, Q\Psi_{_X}+\thre)- \int Q \Delta\Psi_{_X} (\rho\, Q\Psi_{_X}+\thre) \\
& - \int Q^3\Psi_{_X}(\Psi_{_X}^2-1) ( \rho Q\Psi_{_X} + \thre),\\
B_{NL}^1(\bh)&=\frac 12 \int |\rho \, Q\Psi_{_X} + \th |^2   +  \frac 12 \int \left| \nabla ( \rho\, Q\Psi_{_X}+\th ) \right|^2, \\ 
B_{NL}^2(\bh)&= - \frac 12 \int Q^2\Psi_{_X}^2 ( 3(\rho \, Q\Psi_{_X}+\thre)^2+\thima^2)- \int Q\Psi_{_X} | \rho\, Q\Psi_{_X}+\th |^2( \rho \, Q\Psi_{_X}+\thre) \\
&- \frac 14 \int | \rho \, Q\Psi_{_X}+\th|^4.
\end{align*}
Next, we estimate each term. Using the fact that $\nabla\Psi, \; \Delta \Psi$ and $(\Psi^2-1)$ have compact supports and Lemma \ref{compac-supp-argu}, we obtain 
 \begin{align*}
B_L(\bh)&= -  \int (2 \nabla Q \cdot \nabla\Psi_{_X} + Q \Delta\Psi_{_X}) (\rho \, Q\Psi_{_X}+\thre) - \int Q^3\Psi_{_X}(\Psi_{_X}^2-1) ( \rho Q\Psi_{_X} + \thre) \\  
&= O\left( |\rho| \ee + \e \left\| \bh \right\|_{H^1} \right)+  O\left(|\rho| \eeee + \left\| \bh \right\|_{H^1} \eee \right) .
\end{align*}
Using the orthogonality condition \eqref{rho-ortho-DQh}, we get 
\begin{align*}
B_{NL}^1(\bh)&=\frac{\rho^2}{2} \int  Q^2\Psi_{_X}^2 + \rho \int Q\Psi_{_X}  \, \thre + \frac 12 \int |\th|^2+\frac{ \rho^2}{2} \int \left| \nabla(Q\Psi_{_X} ) \right|^2 + \rho  \int \nabla( Q\Psi_{_X}) \cdot \nabla \thre  \\
&+ \frac 12 \int |\nabla \th |^2  =\rho \int Q\Psi_{_X} \thre + \frac 12 \int |\bh|^2  + \frac 12 \int |\nabla \bh |^2+  O(|\rho|^2),
\end{align*}
\begin{align*}
B_{NL}^2(\bh)&= - \frac 12 \int Q^2\Psi_{_X}^2( 3\thre^2+\thima^2)- \frac 14 \int |\th|^4- \rho \int Q\Psi_{_X} |\th|^2 \thre- \int Q\Psi_{_X} |\th|^2 \thre \\& - \frac{ \rho^2}{2} \int Q^2\Psi_{_X}^2 ( 3\thre^2+\thima^2)- \rho \int Q^2\Psi_{_X}^2 |\th|^2 - 2 \rho \int Q^2\Psi_{_X}^2 \,  \thre^2 - \rho^3 \int Q^3\Psi_{_X}^3  \, \thre \\
&- 3 \rho^2 \int Q^3\Psi_{_X}^3 \, \thre -3 \rho \int Q^3\Psi_{_X}^3  \, \thre - \frac{ \rho^4}{4} \int Q^4\Psi_{_X}^4 -\rho^3 \int Q^4\Psi_{_X}^4 - \frac{ 3 \rho^2}{2} \int Q^4\Psi_{_X}^4 .
 \end{align*}
By the equation \eqref{eq_Q} and using again the orthogonality condition \eqref{rho-ortho-DQh}, we have 
\begin{align*}
-3\rho \int Q^3\Psi_{_X}^3 \,\thre  &= -3 \rho \int Q\Psi_{_X} \thre -3 \rho \int (Q-\Delta Q) \, \Psi_{_X}^2(\Psi_{_X}-1)\thre 
  -6 \rho \int \nabla Q . \nabla\Psi_{_X}  \, \thre \\ &- 3 \rho \int \Delta\Psi_{_X} Q \,  \thre  \\&=-3 \rho \int Q\Psi_{_X} \,  \thre + O\left( |\rho| \e \left\| \bh \right\|_{H^1} \right). 
\end{align*}
Using the facts that 
\begin{align*}
&   \rho \int Q\Psi_{_X} |\th|^2 \thre=O(|\rho| \left\| \bh \right\|_{H^1}^3   ) , \\
&   \frac{\rho^2}{2} \int Q^2\Psi_{_X}^2 ( 3\thre^2+\thima^2)- \rho \int Q^2\Psi_{_X}^2 |\th|^2 - 2 \rho \int Q^2\Psi_{_X}^2 \thre^2=O(|\rho|^2 \left\| \bh \right\|_{H^1}^2+ |\rho| \left\| \bh \right\|_{H^1}^2) ,\\ 
&  -  \rho^3 \int Q^3\Psi_{_X}^3 \thre - 3 \rho^2 \int Q^3\Psi_{_X}^3 \thre =O(|\rho|^3 \left\| \bh \right\|_{H^1}+ |\rho|^2 \left\| \bh \right\|_{H^1}) , \\
&\text{and} \\ 
& - \frac{ \rho^4}{4} \int Q^4\Psi_{_X}^4 -\rho^3 \int Q^4\Psi_{_X}^4 - \frac{ 3 \rho^2}{2} \int Q^4\Psi_{_X}^4 =O(|\rho|^4+|\rho^2|),
\end{align*}
we obtain
\begin{align*}
B_{NL}^2&= - \frac 12 \int Q^2\Psi_{_X}^2 ( 3\thre^2+\thima^2) - \int Q\Psi_{_X} |\th|^2\thre   - \frac 14 \int |\bh|^4 -3\rho \int Q\Psi_{_X} \thre \\ &+ O\left( |\rho| \left\| \bh \right\|_{H^1}^2+ +  |\rho| \e \left\| \bh \right\|_{H^1}   + |\rho|^2  \right). 
\end{align*}

Thus, 
\begin{multline}
\label{sum-BL_BNL}
B_L(\bh)+B_{NL}^1(\bh)+B_{NL}^2(\bh)= \frac 12 \int |\bh|^2  - \frac 12 \int Q^2\Psi_{_X}^2 ( 3\thre^2+\thima^2)  + 
\frac 12 \int |\nabla \bh |^2 \\ \qquad  \qquad  \qquad  \qquad  \qquad   \quad  - \frac 14 \int |\bh|^4 
- \int Q\Psi_{_X} |\th|^2\thre  -2\rho \int Q\Psi_{_X} \thre \\= O\bigg( |\rho| \left\| \bh \right\|_{H^1}^2+ |\rho|^2   +  \ee+   \e \left\| \bh \right\|_{H^1} \bigg).  \quad \;     
\end{multline}

Recall that, from \eqref{Quad-form-Phi} we have
 $$\Phi_{X}(\bh)= \frac 12 \int | \nabla \bh |^2 - \frac 12 \int Q^2\Psi_{_X}^2 (3\bh_1^2+\bh_2^2)+ \frac{1}{2} \int |\bh|^2. $$
By \eqref{sum-BL_BNL}, one can see that, 
\begin{align*} \Phi_{X}(\th)&= \frac 14 \int |\bh|^4 + \int Q\Psi_{_X} |\th|^2\thre +2\rho \int Q\Psi_{_X} \thre \\ &+ O\left( |\rho| \left\| \bh \right\|_{H^1}^2+ |\rho|^2 +  \ee+   \e \left\| \bh \right\|_{H^1}  \right).
\end{align*}
Thus, 
\begin{align*}
 \left| \Phi_{X}(\th) \right|   &\leq C \bigg( \left\| \bh \right\|_{H^1}^3 + 2 | \rho|  \left| \int Q\Psi_{_X} \thre  \right|+|\rho|^2 +  \ee+   \e \left\| \bh \right\|_{H^1}     \bigg) .
\end{align*}
By the coercivity property \eqref{coercivity-Psi}, we obtain
$$\left\| \bh \right\|_{H^1}=O\left( |\rho|+ \dt^{\frac 32}+ \e + \left| \int Q\Psi_{_X} \thre \right| \right)  .$$
By \eqref{rho+Qh1}, we deduce
\begin{equation}
    \label{estimate-on-h}
   \left\|h \right\|_{H^1_0(\Omega)}= \left\| \bh \right\|_{H^1(\R^3)}=  O\left( |\rho|+ \dt^{\frac 32}+ \e   \right),
\end{equation}
and thus, by \eqref{rho+delta}, we get
$$ \dt=O\left( | \rho|+ \e \right),$$ which implies \eqref{equiv} and concludes the proof of Proposition \ref{lem-equiv-modu-param}. 
\end{itemize}
\end{proof}

\begin{lem}
\label{derv-of-mod-par}
Under the assumptions of Proposition \ref{lem-equiv-modu-param}, for all $t \in D_{\delta_0},$ we have  
\begin{equation}
    |\rho{'}(t)|+ |X{'}(t)|+ |\theta{'}(t)|=O\left(\delta + \e \right).
\end{equation}
\end{lem}
\begin{proof}
Let $\ds(t):= \delta(t)+ |\rho{'}(t)|+ |X{'}(t)|+ |\theta{'}(t)|. $ Using the \NNls equation, Lemma \ref{compac-supp-argu}, Proposition \ref{lem-equiv-modu-param} and the Sobolev embedding $H^1_0(\Omega) \subset L^6(\Omega)$, we obtain
\begin{multline}
\label{eq-h+modu-para}
    i \partial_{t} h + \Delta h + i \rho' Q_{_{-X} }\Psi-i X' \cdot \nabla Q_{_{-X}} \Psi - \theta' \, Q_{_{-X}} \Psi \\= O \left(\delta + \e + \ds (\delta + \e) \right) \; \text{in}\; L^2 .
\end{multline}
By the orthogonality conditions \eqref{rho-ortho-DQh}, \eqref{modu-ortho-dQh-im}, \eqref{modu-ortho-dQh} and Proposition \ref{lem-equiv-modu-param}, we have 
\begin{equation}
\label{dt-modu-ortho-dQh-im}
\im \int_{\Omega} \partial_t h \,  Q_{_{-X}} \Psi dx= \im \int_{\Omega} h \, X^{'} \cdot \nabla Q_{_{-X}} \Psi \, dx =O\left( \ds (\delta +  \e ) \right),
\end{equation}
\begin{multline}
\label{dt-modu-ortho-dQh}
\re \int_{\Omega} \partial_t h\,  \partial_{x_k} (Q_{_{-X}} \Psi) dx=\sum_{j=1}^3 \re \int_{\Omega} h \, X^{'}_j   (\partial_{x_k}( \partial_{x_j} Q \Psi) )\, dx \\ =O\left( \ds (\delta +  \e ) \right), \; k=1,2,3,
\end{multline}
\begin{multline}
\label{dt-rho-ortho-DQh}
\re \int_{\Omega}  \partial_t h \,  \Delta (Q_{_{-X}} \Psi)  dx= \sum_{j=1}^3 \re \int_{\Omega} h \, X^{'}_j \Delta (\partial_{x_j} Q_{_{-X}} \Psi) \, dx = O\left( \ds (\delta +  \e ) \right).
\end{multline}
Multiplying \eqref{eq-h+modu-para} by $Q_{_{-X}}\Psi,$ integrating the real part, using \eqref{dt-modu-ortho-dQh-im} and then integrating by parts, we get 
\begin{equation}
    \label{derv-t-theta}
    |\theta'|= O \left(\delta + \e + \ds (\delta +  \e ) \right) .
\end{equation}
Similarly, multiplying \eqref{eq-h+modu-para} by $\partial_{x_j}(Q_{_{-X}}\Psi),\, j\in{1,2,3},$ integrating the imaginary part, using \eqref{dt-modu-ortho-dQh} and Proposition \ref{lem-equiv-modu-param},  we obtain 
\begin{equation}
\label{derv-t-X}
    |X_j'(t)|= O \left(\delta + \e+ \ds (\delta + \e) \right), \quad  \; j=1,2,3.
\end{equation}

Multiplying \eqref{eq-h+modu-para} by $\Delta(Q_{_{-X}}\Psi),$ integrating the imaginary part, and using \eqref{dt-rho-ortho-DQh} and Proposition \ref{lem-equiv-modu-param}, we get 
\begin{equation}
    \label{derv-t-Rho}
    |\rho'|= O \left(\delta + \e+ \ds (\delta + \e ) \right).
\end{equation}

Summing up \eqref{derv-t-theta}, \eqref{derv-t-X} and \eqref{derv-t-Rho}, we obtain 
$$\ds=O\left(\delta + \e + \ds (\delta + \e) \right) ,$$
which concludes the proof by choosing $\delta_0$ sufficiently small. 
\end{proof}

\section{Scattering}
\label{sec-scattering}
In this section, we prove Theorem \ref{main-theo}. We start by proving, in \S \ref{Compactness properties}, that the extension $\uu$ of a non-scattering solution $u(t)$ to the \NNls equation, satisfying \eqref{u<Qthro} and \eqref{MEuu=MEQQTheo}, is compact in $H^1(\R^3)$, up to a spatial translation parameter $x(t).$ In \S \ref{sec-bddxx}, we prove that $x(t)$ is bounded using an auxiliary translation parameter (obtained by ignoring the obstacle), a local virial identity and the estimates from Section \ref{Sec-Modulation} for the modulation parameters. In \S \ref{Sec-convergence-mean}, we prove that the parameter $\delta(t)$ converges to $0$ in mean. Finally, combining the compactness properties with the control of the space translation parameter $x(t)$ and the convergence in mean, we obtain a contradiction from the existence of a non-scattering solution, thus, concluding the proof of Theorem \ref{main-theo}.

\subsection{Compactness properties}
\label{Compactness properties}
\begin{prop}
\label{compac-prop}
Let $u(t)$ be a solution of (NLS$_\Omega$) such that
\begin{equation}
    M[u]=M_{\R^3}[Q], \quad E[u]=E_{\R^3}[Q] \quad \text{and} \quad  \left\| u_0 \right\|_{L^2(\Omega)} < \left\| \nabla Q \right\|_{L^2(\R^3)},
\end{equation}
which does not scatters in positive time. Then there exists a continuous function $x(t)$ such that \begin{equation}
\label{def-K}
    K= \{ \uu(x+x(t),t),\, t \in [0, + \infty) \} 
\end{equation}
has a compact closure in $H^1(\R^3).$
\end{prop}
\begin{proof}
It is sufficient to show that for every sequence of time $\tau_n \geq 0,$ there exists (extracting if necessary) a sequence $(x_n)_n$ such that $u(x+x_n,\tau_n)$ has a limit in $H^1_0(\Omega)$. \\
By the profile decomposition in Theorem \ref{profile decom}, we have 
\begin{equation}
\label{decomp of u_n}
    u_n:=u(x,\tau_n)= \sum_{j=1}^J  \phi_n^j (x)+ \omega_{n}^J(x),
\end{equation}
where $\phi_n^j$ are defined in Theorem \ref{profile decom}, and $\omega_n^J$ satisfies \eqref{limJ limsup_n omega}.
We need to show that $J^*=1,\omega_n^1 \rightarrow 0   $ in $H^1_0(\Omega),$ and $ t_n^j \equiv 0$. By the Pythagorean expansion properties of the profile decomposition we have 
\begin{align}
\label{mass-profile}
\sum_{j=1}^J \lim_{n \rightarrow \infty} M[\phi_n^j] + \lim_{n \rightarrow \infty} M[\omega_n^J]= \lim_{n\rightarrow \infty} M[u_n]=M[Q],   \\
\label{energy-profile}
\sum_{j=1}^J \lim_{n \rightarrow \infty} E[\phi_n^j] + \lim_{n \rightarrow \infty} E[\omega_n^J]= \lim_{n \rightarrow \infty} E[u_n] =E[Q] .
\end{align}
We consider two possibilities. \\ 
\underline{\bf{Scenario \rm{I}}:} More than one profile are nonzero, i.e., $J^{*} \geq 2.$ Thus, there exists an $\varepsilon >0$ such that for all $j$, \begin{align}
\label{M&E profil}
   M[\phi_n^j] E[\phi_n^j] &\leq M_{\R^3}[Q] E_{\R^3}[Q] - \varepsilon, \\
\label{profil-L^2-bound}
   \left\| \phi_n^j \right\|_{L^2(\Omega)} \left\| \nabla \phi_n^j \right\|_{L^2(\Omega)} &\leq \left\| Q \right\|_{L^2(\R^3} \left\| \nabla Q \right\|_{L^2(\R^3)} - \varepsilon.
\end{align}

Recall that by  \cite[Theorem 3.2]{KiVisnaZhang16}, if $v_0\in H^1_0(\Omega)$ satisfies 
\begin{align}
\label{grad-w}
\left\| v_0 \right\|_{L^2(\Omega)} \left\| \nabla v_0 \right\|_{L^2(\Omega) }& <  \left\| Q \right\|_{L^2(\R^3)} \left\| \nabla Q \right\|_{L^2(\R^3)},  \\
\label{MEw_0}
M[v_0] E[v_0] &< M_{\R^3}[Q]E_{\R^3}[Q],
\end{align}
 then the corresponding solution $v(t)$ of \Nls scatters in both time directions.  \\

$\bullet$ Suppose $j$ is as in Case 1 (Theorem \ref{profile decom}), i.e., $x_n^j =0$ for all $n:$ \\ 
When $t_n^j \equiv 0 $, we define $v^j$ as the solution to (NLS$_\Omega$) with initial data $v^j(0)=\phi^j.$ \\ 
When $t_n^j \rightarrow \pm \infty $, we define $v^j$ as the solution to {\rm{(NLS$_\Omega$)}}, which scatters to $e^{it\Delta_\Omega} \phi^j$ as $t\rightarrow \pm \infty:$   
$$   \lim_{t \rightarrow \pm \infty} \left\| v^j(t) - e^{it\Delta_\Omega} \phi^j \right\|_{H^1_0(\Omega)}=0  .$$
In both cases, we have 
 \begin{equation}
\label{vj(0)-phi ^j}
\lim_{n \rightarrow \infty} \left\| v^j(t_n^j) - \phi_n^j \right\|_{H^1_0(\Omega)}=0.     
\end{equation}

Thus, by \eqref{M&E profil} and \eqref{profil-L^2-bound}, $v^j$ satisfies \eqref{grad-w} and \eqref{MEw_0}, and we see that $v^j$ is a global solution with finite scattering size. Therefore, we can approximate $v^j$ in $L^5 H^{1,\frac{30}{11}}(\R \times \Omega)$ by $C^\infty_c( \R \times \R^3)$ functions. More precisely, for any $\varepsilon > 0,$ there exists $ \psi_{\varepsilon}^{j} \in C^\infty_c ( \R \times \R^3)$ such that 
$$\left\| v^j- \psi_{\varepsilon}^{j} \right\|_{L^5 H^{1,\frac{30}{11}}(\R \times \Omega)} \leq \frac{\varepsilon}{2}. $$
Let $v_n^j(t,x)=v^j(t+t_n^j,x).$ Then from above $v_n^j$ is a global and scattering solution and by changing variables in time, for any $\varepsilon > 0,$ there exists $ \psi_{\varepsilon}^{j} \in C^\infty_c ( \R \times \R^3)$ such that, for $n$ sufficiently large, we have 

\begin{equation}
\label{aprox-v_n1}
\left\| v_n^j(t,x)-  \psi_\varepsilon^j(t+t_n^j,x) \right\|_{L^5 H^{1,\frac{30}{11}}(\R \times \Omega)} < \varepsilon.
\end{equation}

$\bullet$ Suppose $j$ is as in Case 2 (Theorem \ref{profile decom}): \\      
We apply Theorem \ref{contruction of v_n if d(x_n)} to obtain a global solution $v_n^j$ with $v_n^j(0)= \phi_n^j$. Furthermore, this solution has finite scattering size and satisfies, for $n$ sufficiently large, \\ 
\begin{equation}                    
\label{aprox-v_n2}
\left\| \vv_n^j(t,x)-  \psi_\varepsilon^j(t+t_n^j,x-x_n^j) \right\|_{\esp} < \varepsilon.
\end{equation}

In all cases, we can find $\psi_\varepsilon^j \in C^{\infty}_c $ such that \eqref{aprox-v_n2} holds, and there exists $C_j>0$, independent of $n$, such that
 \begin{equation}
\label{v_n_j bound X^1}
\|v_n^j \|_{X^1(\R \times \Omega )} \leq C_j \,  .
\end{equation}
Note that for large $j,$ by the small data theory, we have $\|v_n^j\|_{X^1(\R \times \Omega)} \lesssim \| \phi_n^j\|_{H^1_0(\Omega)}$. \\ 
Combining this with \eqref{mass-profile}, \eqref{energy-profile}, we deduce 
\begin{equation}
\label{sum-nonlinear-prof-carre}
\limsup_{ n \rightarrow + \infty} \sum_{j=1}^J \left\| v_n^j \right\|^2_{X^1(\R  \times \Omega)} \leq C \quad \text{ uniformly for finite} \;  J \leq J^\star.
\end{equation}

We first prove the asymptotic decoupling of the nonlinear profile, using the orthogonality properties \eqref{asymptotic ortho para}.

\begin{lem}[Decoupling of nonlinear profiles]
\label{Decoup-nonlinear-prof}
 For $k\neq j,$ we have
 \begin{multline}
\label{v^j v^k ortho}
\lim_{n \rightarrow + \infty} \left\|v_n^j v_n^k \right\|_{L^{\frac{5}{2}}  H^{1,\frac{15}{11}}_0(\R \times \Omega)}+  \left\|  \nabla v_n^j  \nabla v_n^k \right\|_{L^{\frac{5}{2}} L^{\frac{15}{11}}(\R \times \Omega)}  \\ +     \left\| v_n^j v_n^k \right\|_{L^{\frac{5}{2}} L^{\frac{30}{17}} (\R \times \Omega)} + \left\| \nabla v_n^j v_n^k \right\|_{L^{\frac{5}{2}} L^{\frac{30}{17}} (\R \times \Omega)} = 0.
\end{multline}
\end{lem}
\begin{proof}
We only prove $\| v_n^j v_n^k \|_{L^{\frac{5}{2}}H^{1,\frac{15}{11}}_0  (\R \times \Omega)}+   \| v_n^j v_n^k \|_{L^{\frac{5}{2}} L^{\frac{30}{17}} (\R \times \Omega)}     =o_n(1) $. The other proofs are analogous. 
Recall that by \eqref{aprox-v_n2}, for any $\varepsilon >0,$ there exists $N_\varepsilon \in \N$ and $\psi_\varepsilon^j,\psi_\varepsilon^k \in C^\infty_c(\R \times \R^3) $ such that for all $n \geq \N_\varepsilon$ we have  
\begin{multline}
\label{v^k-v^j}
\left\| \vv_n^k(t,x)-  \psi_\varepsilon^k(t+t_n^k,x-x_n^k) \right\|_{\esp}  \\ + \left\| \vv_n^j(t,x)-  \psi_\varepsilon^j(t+t_n^j,x-x_n^j) \right\|_{\esp} < \varepsilon.  
\end{multline}

Using \eqref{asymptotic ortho para}, one can see that the supports of $\psi_{\varepsilon}^{j}(t,x)$ and $\psi^k_{\varepsilon}( \cdot+t_n^k-t_n^j, \cdot-x_n^k+x_n^j)$ are disjoint for $n$ sufficiently large ( if $j,k$ as in Case $1,$ then $\psi_{\varepsilon}^{j}(\cdot,\cdot)$ and $\psi^k_{\varepsilon}( \cdot+t_n^k-t_n^j, \cdot)$ have disjoint time supports), and similarly, for the derivatives. Hence,
\begin{align}
\label{psi^k-psi^j}
\lim_{n \rightarrow + \infty} \left\| \psi_{\varepsilon}^{j}(t,x) \, \psi^k_{\varepsilon}( \cdot+t_n^k-t_n^j,\cdot-x_n^k+x_n^j) \right\|_{L^\frac{5}{2}H^{1,\frac{15}{11}} (\R \times \R^3)}&=0,  \\
\label{psi^k-psi^j-2}
\lim_{n \rightarrow + \infty} \left\| \psi_{\varepsilon}^{j}(t,x) \, \psi^k_{\varepsilon}( \cdot+t_n^k-t_n^j,\cdot-x_n^k+x_n^j) \right\|_{L^\frac{5}{2}H^{1,\frac{30}{17}}(\R \times \R^3)}&=0.
\end{align}

Combining \eqref{v^k-v^j}, \eqref{psi^k-psi^j} and \eqref{v_n_j bound X^1}, we have
\begin{align*}
\left\|v_n^j v_n^k \right\|_{L^{\frac{5}{2}}H^{1,\frac{15}{11}}_0 (\R \times \Omega)} &\leq   \left\| \vv_n^j-\psi_\varepsilon^j(\cdot+t_n^j,\cdot-x_n^j) \right\|_{L^5 H^{1,\frac{30}{11}} (\R\times \R^3)  } \Vert \vv_n^k  \rVert_{\esp} \\ 
&+ \left\| \psi_{\varepsilon}^j \right\|_{\esp} \left\| \vv_n^k-\psi_\varepsilon^k(\cdot+t_n^k,\cdot-x_n^k) \right\|_{\esp} \\ 
&+\left\| \psi_{\varepsilon}^{j}(t,x) \, \psi^k_{\varepsilon}( \cdot+t_n^k-t_n^j,\cdot-x_n^k+x_n^j) \right\|_{L^\frac{5}{2}H^{1,\frac{15}{11}}(\R \times \R^3)} \leq  C\varepsilon,
\end{align*}
provided $n$ is large enough, since the last term goes to $0$ as $n$ goes to infinity.

Next, we estimate $  \left\| v_n^j v_n^k \right\|_{L^{\frac{5}{2}} L^{\frac{30}{17}} (\R \times \Omega)} $ as follows  
\begin{align*}
\left\|v_n^j v_n^k \right\|_{L^{\frac{5}{2}}L^{\frac{30}{17}} (\R \times \Omega)} &\leq   \left\| \vv_n^j-\psi_\varepsilon^j(\cdot+t_n^j,\cdot-x_n^j) \right\|_{L^5_{t,x} (\R\times \R^3)  } \Vert v_n^k  \rVert_{L^5 L^{\frac{30}{11}} (\R \times \R^3)} \\ 
&+ \left\| \psi_{\varepsilon}^j \right\|_{L^5 L^{\frac{30}{11}}} \left\| \vv_n^k-\psi_\varepsilon^k(\cdot+t_n^k,\cdot-x_n^k) \right\|_{L^5_{t,x} (\R\times \R^3)  }  \\ 
&+\left\| \psi_{\varepsilon}^{j}(t,x) \, \psi^k_{\varepsilon}( \cdot+t_n^k-t_n^j,\cdot-x_n^k+x_n^j) \right\|_{L^\frac{5}{2}H^{1,\frac{30}{17}}(\R \times \R^3)} .
\end{align*}
Using \eqref{v^k-v^j}, \eqref{psi^k-psi^j-2} and \eqref{v_n_j bound X^1} and Sobolev embedding $ \left\| \cdot  \right\|_{L^5_{t,x}}   \leq C \left\| \cdot \right\|_{L^5 H^{1,\frac{30}{11}}}, $ we obtain that, for large $n$,
\begin{align*}
\left\|v_n^j v_n^k \right\|_{L^{\frac{5}{2}}L^{\frac{30}{17}} (\R \times \Omega)}  \leq  C\varepsilon,
\end{align*}
provided $n$ is large enough, which concludes the proof of Lemma \ref{Decoup-nonlinear-prof}.
\end{proof}

We return to the proof of Proposition \ref{compac-prop}. As a consequence of the asymptotic decoupling of the nonlinear profile in Lemma \ref{Decoup-nonlinear-prof},
we have 
 \begin{equation}
\label{limsup sum v^n_j}
    \limsup_{n \rightarrow \infty} \| \sum_{j=1}^{J} v_n^j \|_{X^1(\R \times \Omega)} \leq C 
\end{equation}
uniformly for finite $J \leq J^*.$ Indeed, by \eqref{sum-nonlinear-prof-carre} and \eqref{v^j v^k ortho} we obtain
\begin{align*}
    \left\| \sum_{j=1}^J v_n^j \right\|_{L^5 L^{\frac{30}{11}} ( \R \times \Omega)  }^2 &=  \left\|\left( \sum_{j=1}^J v_n^j\right)^2 \right\|_{L^{\frac{5}{2}}L^{\frac{15}{11}}(\R \times \Omega)}  \\ &\leq \sum_{j=1}^{J} \left\|v_n^j\right\|_{L^5_t L^{\frac{30}{11}}_x(\R \times \Omega)}^2 + C(J) \sum_{j\neq k } \left\| v_n^j v_n^k \right\|_{L^{\frac{5}{2}}L^{\frac{15}{11}}(\R \times \Omega)}\\ &\leq  C +o_n(1).
\end{align*}
Similarly, 

\begin{align*}
 \left\| \sum_{j=1}^J \nabla v_n^j \right\|_{L^5 L^{\frac{30}{11}} (\R \times \Omega)}^2 &=  \left\|\left( \sum_{j=1}^J  \nabla v_n^j\right)^2 \right\|_{L^{\frac{5}{2}}L^{\frac{15}{11}}(\R \times \Omega)} \\ &\leq \sum_{j=1}^{J} \left\| \nabla v_n^j\right\|_{L^5_t L^{\frac{30}{11}}(\R \times \Omega)}^2 + C(J) \sum_{j\neq k } \left\| \nabla v_n^j  \nabla v_n^k \right\|_{L^{\frac{5}{2}}L^{\frac{15}{11}} (\R \times \Omega)}  \leq  C. 
\end{align*}

This completes the proof of \eqref{limsup sum v^n_j}. Using similar argument, one can check that for given $\eta>0,$ there exists $J^{'}:=J^{'}(\eta)$ such that 
\begin{equation}
\label{sumj-j'}
\forall J \geq J^{'} , \quad \limsup_{n \rightarrow \infty} \| \sum_{j=J^{'}}^{J} v_n^j \|_{ X^1 ( \R \times \Omega )} \leq \eta  .
\end{equation}

For each $n$ and $J$, we define an approximate solution $u_n^J$ to (NLS$_\Omega$) by
\begin{equation}
\label{approx decom uJ}
    u_n^J= \sum_{j=1}^{J}v_n^j+ e^{it \Delta_{\Omega}} \omega_{n}^J   .
\end{equation} 
Before continuing with the rest of the proof of Proposition \ref{compac-prop}, we claim that the following statements hold true.
\begin{clm}
\label{u_n^J(0)-u_n(0)}
$$\lim_{n \rightarrow \infty } \left\| u_n^J(0)-u_n(0) \right\|_{H^1_0(\Omega)}=0 .$$
\end{clm}
\begin{clm}
\label{limsup of uJ}
$$\exists C>0, \; \forall J, \quad  \limsup_{n \rightarrow \infty} \left\|u_n^J \right\|_{X^1(\R \times \Omega)} \leq C \,. $$
\end{clm}
\begin{clm}
\label{the euror}
$$\lim_{J \rightarrow J^*} \limsup_{n \rightarrow \infty} \left\| i\partial_t u_n^J+ \Delta_\Omega u_n^J + \left|u_n^J\right|^2u_n^J \right\|_{N^1(\R)}=0 ,$$
with $N^1$ defined in \eqref{def-N1}.
\end{clm}
Applying Lemma \ref{stability}, we get that $u_n$ is a global solution with finite scattering size, which yields a contradiction by showing that there is only one profile. Hence, Scenario $\rm{I}$ cannot occur. \\ 

\begin{proof}[Proof of Claim \ref{u_n^J(0)-u_n(0)} ]
Using \eqref{vj(0)-phi ^j}, if $j$ is as in Case $1,$ or the fact that $v_n^j(0)=\phi^j_n$ if $j$ is as in Case $2,$ together with the decomposition of $u_n$ in  \eqref{decomp of u_n} and $u^J_n$ in  \eqref{approx decom uJ}, we obtain \begin{equation}
    \left\|u_n^J(0)-u_n(0)\right\|_{H^1_0(\Omega)} \leq \sum_{j=1}^{J} \left\| v_n^j(0)-\phi_n^j \right\|_{H^1_0(\Omega)} \longrightarrow 0\;  \text{ as } n \rightarrow \infty. 
\end{equation}
\end{proof}

\begin{proof}[Proof of Claim \ref{limsup of uJ} ]
Using \eqref{limsup sum v^n_j}, Strichartz estimate \eqref{strich-S1}  with \eqref{limJ limsup_n omega}, we obtain 
\begin{align*}
 \limsup_{n \rightarrow \infty} \left\|u_n^J \right\|_{X^1(\R \times \Omega)} 
\leq     \limsup_{n \rightarrow \infty} \| \sum_{j=1}^{J} v_n^j \|_{X^1(\R \times \Omega)} + \limsup_{ n \rightarrow + \infty} \left\| \omega_n^J \right\|_{H^1_0(\Omega)} \leq C.
\end{align*}
\end{proof}

\begin{proof}[Proof of Claim \ref{the euror} ]
Let $F(z)= -|z|^2z$, recall $\sum_{j=1}^J v_n^j= u_n^J -e^{it\Delta_\Omega} \omega_n^J ,$ and write
 \begin{align*}
    (i\partial_t + \Delta_\Omega)u_n^J - F(u_n^J)&= \sum_{j=1}^{J} F(v_n^j)- F(u_n^J) \\ &= \sum_{j=1}^{J} F(v_n^j)- F(\sum_{j=1}^J v_n^j)+F(u_n^J-e^{it\Delta_\Omega} \omega_n^J)-F(u_n^J).
    \end{align*} 
    We have 
    \begin{align}
    \left|  \sum_{j=1}^J F(v_n^j)-F(\sum_{j=1}^J v_n^j)  \right| \leq C  \sum_{j \neq k} | v_n^j |^2 |v_n^k|. 
    \end{align}
    Taking the derivatives, we get
    \begin{align*} 
     \left| \nabla \bigg\{ \sum_{j=1}^J F(v_n^j)-F(\sum_{j=1}^J v_n^j)  \bigg\} \right| \leq  C \sum_{j \neq k} | \nabla v_n^j | |v_n^j| |v_n^k|+  C  \sum_{j \neq k} |v_n^j|^2 | \nabla v_n^k |,
     \end{align*}
   which yields 
    \begin{align*}
        \left\| \sum_{j=1}^J F(v_n^j)-F(\sum_{j=1}^J v_n^j) \right\|_{{L^{\frac{5}{3}}L^{\frac{30}{23}}} } 
        &\leq C \left( \sum_{j \neq k} \left\| v_n^j \right\|_{L^5_{t,x}} \left\| v_n^j v_n^k \right\|_{L^{\frac{5}{2}}L^{\frac{30}{17}}}    \right) ,  \\  \\
 \left\| \nabla \bigg\{ \sum_{j=1}^J F(v_n^j)-F(\sum_{j=1}^J v_n^j) \bigg\} \right\|_{{L^{\frac{5}{3}}L^{\frac{30}{23}}} } 
 &\leq C \bigg( \sum_{j \neq k} \left\|  v_n^j v_n^k \nabla v_n^j \right\|_{L^{\frac{5}{3}} L^{\frac{30}{23}}} + \sum_{j \neq k} \left\| |v_n^j|^2 \nabla v_n^k \right\|_{L^{\frac{5}{3}} L^{\frac{30}{23}}}  \bigg) \\ 
   &\leq C \sum_{j \neq k} \left\| v_n^j \right\|_{L^5_{t,x}} \left(  \left\| \nabla v_n^j v_n^k \right\|_{L^{\frac{5}{2}}L^{\frac{30}{17} }}  +  \left\| v_n^j \nabla v_n^k \right\|_{L^{\frac{5}{2}}L^{\frac{30}{17}}} \right),
    \end{align*}
    which goes to $0$ as $n \rightarrow \infty$, in view of Lemma \ref{Decoup-nonlinear-prof} and \eqref{v_n_j bound X^1}. \\ 
    
    In addition, 
    \begin{align}
    \label{f(u-e^(it)omega)-f(u)}
        \left\|F(u_n^J-e^{it\Delta_\Omega} \omega_n^J )-F(u_n^J) \right\|_{L^{\frac{5}{3}}  H^{1,\frac{30}{23}}   }
  &  \leq \left\| F(u_n^J-e^{it\Delta_\Omega} \omega_n^J )-F(u_n^J) \right\|_{L^{\frac{5}{3}}L^{\frac{30}{23}}} \\ \label{f'(u-e^(it)omega)-f'(u)} & + \left\| \nabla \left( F(u_n^J-e^{it\Delta_\Omega} \omega_n^J )-F(u_n^J) \right) \right\|_{L^{\frac{5}{3}}L^{\frac{30}{23}}} .
    \end{align}
    We estimate the differences as 
    \begin{align*}
    \left| F(u_n^{J} -  e^{ i t \Delta_{\Omega}} \omega_n^J) -F(u_n^J) \right| &\leq C \left(  \left| e^{ i t \Delta_{\Omega}} \omega_n^J \right|^2  \left| e^{ i t \Delta_{\Omega}} \omega_n^J  \right|+ \left| u_n^J   \right|^2  \left| e^{ i t \Delta_{\Omega}} \omega_n^J  \right| \right) ,\\
        \left| \nabla \left\{ F(u_n^{J} -  e^{ i t \Delta_{\Omega}} \omega_n^J) -F(u_n^J)  \right\} \right|  &\leq  C \bigg(      \left|  e^{ i t \Delta_{\Omega}} \omega_n^J \right|^2 \left| \nabla e^{ i t \Delta_{\Omega}} \omega_n^J \right| 
         + \left| \nabla u^J_n \right|  \left| u^J_n \right|   \left| e^{ i t \Delta_{\Omega}} \omega_n^J  \right| \\
         &+   \left| \nabla u_n^J\right| \left| \nabla e^{ i t \Delta_{\Omega}} \omega_n^J  \right|^2+ \left|   u^J_n \right|^2 \left|  \nabla e^{ i t \Delta_{\Omega}} \omega_n^J  \right|    \bigg).
    \end{align*}
    Using Claim \ref{limsup of uJ}, H\"older and Sobolev inequalities, we get 
    \begin{align*}
        \eqref{f(u-e^(it)omega)-f(u)} &\leq \left\| e^{it \Delta_{\Omega}} \omega_n^J \right\|_{L^5_{t,x}}
        \left[ \left\| u_n^J \right\|_{ L^5 L^{\frac{30}{11}} } \left\|  u_n^J  \right\|_{L^5_{t,x}}  +  \left\| e^{it \Delta_{\Omega}} \omega_n^J \right\|_{ L^5 L^{\frac{30}{11}} }   \left\| e^{it \Delta_{\Omega}} \omega_n^J \right\|_{ L^5_{t,x}}   \right]      \\
 &\leq \left\| e^{it \Delta_{\Omega}} \omega_n^J \right\|_{L^5_{t,x}} 
        \left[ \left\| u_n^J \right\|_{ X^1 }^2+ \left\| e^{it \Delta_{\Omega}} \omega_n^J \right\|_{ X^1 }^2    \right]   + \left\| e^{it \Delta_{\Omega}} \omega_n^J \right\|_{L^5_{t,x}}^2 \left\| u_n^J \right\|_{X^1}  \\
       &  \leq C \left\| e^{it \Delta_{\Omega}} \omega_n^J \right\|_{L^5_{t,x}} ,
     \end{align*}
    which converges to $0$ as $n \rightarrow \infty$ and $J \rightarrow \infty$. Similarly, 
    \begin{align*}
        \eqref{f'(u-e^(it)omega)-f'(u)} &\leq \left\| \nabla u_n^J u_n^J \right\|_{L^{\frac{5}{2}}L^{\frac{30}{17}}} \left\| e^{it \Delta_{\Omega}} \omega_n^J \right\|_{L^5_{t,x}}+ \left\| \nabla u_n^J \right\|_{L^5 L^{\frac{30}{11}}} \left\| \left|e^{it \Delta_{\Omega}} \omega_n^J \right|^2 \right\|_{L^\frac{5}{2}_{t,x}} \\ & + \left\| \nabla (e^{it \Delta_{\Omega}} \omega_n^J) \right\|_{L^{5}L^{\frac{30}{11}}} \left\| \left|e^{it \Delta_{\Omega}} \omega_n^J \right|^2 \right\|_{L^\frac{5}{2}_{t,x}} +  \left\| u_n^J \right\|_{L^5_{t,x}} \left\| u_n^J \nabla e^{it \Delta_{\Omega}} \omega_n^J \right\|_{L^{\frac{5}{2}} L^{\frac{30}{17}}} \\
        & \leq \left\| \nabla u_n^J \right\|_{L^5 L^{\frac{30}{11}}} \left[ \left\|u_n^J \right\|_{L^5_{t,x}} \left\| e^{it \Delta_{\Omega}} \omega_n^J \right\|_{L^5_{t,x}} + \left\| e^{it \Delta_{\Omega}} \omega_n^J \right\|_{L^5_{t,x}}^2 \right] \\ &+ \left\| \nabla e^{it \Delta_{\Omega}} \omega_n^J \right\|_{L^5 L^{\frac{30}{11}}} \left\| e^{it \Delta_{\Omega}} \omega_n^J \right\|_{L^5_{t,x}}^2 + \left\|u_n^J \right\|_{L^5_{t,x}} \left\| u_n^J \nabla e^{it \Delta_{\Omega}} \omega_n^J\right\|_{L^{\frac{5}{2}} L^{\frac{30}{17}}} .
    \end{align*}
    Thus, it remains to show that 
    \begin{equation}
    \label{remains-term-local-smooth}
        \lim_{J \rightarrow \infty} \limsup_{n \rightarrow \infty} \left\| u_n^J \nabla e^{it \Delta_{\Omega}} \omega_n^J\right\|_{L^{\frac{5}{2}} L^{\frac{30}{17}}}=0 . 
    \end{equation}
    Recall that $u_n^J= \sum_{j=1}^{J}v_n^j+ e^{it\Delta_\Omega} \omega_n^J.$ Then
    \begin{align*}
        \left\| u_n^J \nabla e^{it \Delta_{\Omega}} \omega_n^J\right\|_{L^{\frac{5}{2}} L^{\frac{30}{17}}} &\leq  \left\|  \sum_{j=1}^{J}v_n^j \nabla e^{it\Delta_\Omega} \omega_n^J \right\|_{L^{\frac{5}{2}} L^{\frac{30}{17}}} + \left\| e^{it\Delta_\Omega} \omega_n^J \nabla e^{it\Delta_\Omega} \omega_n^J \right\|_{L^{\frac{5}{2}} L^{\frac{30}{17}}} \\ &\leq  \left\|  \sum_{j=1}^{J}v_n^j \nabla e^{it\Delta_\Omega} \omega_n^J \right\|_{L^{\frac{5}{2}} L^{\frac{30}{17}}}+ \left\| e^{it\Delta_\Omega} \omega_n^J \right\|_{L^5_{t,x}} \left\|\nabla e^{it\Delta_\Omega} \omega_n^J \right\|_{L^{5} L^{\frac{30}{11}}}.
    \end{align*}
    
    Hence, Claim \ref{the euror} holds if \begin{equation*}
        \lim_{J \rightarrow \infty} \limsup_{n \rightarrow \infty}\left\|  \sum_{j=1}^{J}v_n^j \nabla e^{it\Delta_\Omega} \omega_n^J \right\|_{L^{\frac{5}{2}} L^{\frac{30}{17}}}=0 .
    \end{equation*}
    From \eqref{sumj-j'}, we have $\forall \eta >0 , \exists J'=J'(\eta) $ such that
     $$\forall J \geq J^{'} , \quad  \limsup_{n \rightarrow \infty} \left\| \sum_{j=J'}^{J} v_n^j \right\|_{X^1} < \eta .$$
    Thus, we have 
    \begin{align*}
    \limsup_{ n \rightarrow \infty} \left\| \left( \sum_{j=J^{'}}^{J} v_n^j \right) \nabla e^{it \Delta_{\Omega}} \omega_n^J \right\|_{L^{\frac 52} L^\frac{30}{17}} \leq  \limsup_{n  \rightarrow \infty}  \left\| \sum_{j=J^{'}}^{J} v_n^j \right\|_{X^1} \left\|   \nabla e^{it \Delta_{\Omega}} \omega_n^J \right\|_{L^5_{t,x}}      \leq \eta,
     \end{align*}
    where $\eta$ is arbitrary and $J^{'}=J^{'}(\eta)$ as in  \eqref{sumj-j'}. Thus, to prove \eqref{remains-term-local-smooth} it suffices to show that 
    \begin{equation}
    \label{limJ limsup^n v omega'}
    \lim_{J \rightarrow \infty} \limsup_{n \rightarrow \infty}\left\|  v_n^j \nabla e^{it\Delta_\Omega} \omega_n^J \right\|_{L^{\frac{5}{2}} L^{\frac{30}{17}}}=0  \; \; \text{ for all } 1\leq j \leq J' .
    \end{equation}
    We approximate $v_n^j$ by $C^\infty_c(R\times \R^3) $ functions $\psi_\varepsilon^j$ obeying \eqref{aprox-v_n2} with support in $[-T,T] \times \{ |x| \leq R\}.$ From Proposition \ref{local smoothing} and \eqref{limJ limsup_n omega}, we deduce 
    \begin{align*}
        \left\|  v_n^j \nabla e^{it\Delta_\Omega} \omega_n^J \right\|_{L^{\frac{5}{2}} L^{\frac{30}{17}}} &\leq \left\| v_n^j-\psi^j_\varepsilon(\cdot + t_n^j,\cdot-x_n^j)\right\|_{L^5_{t,x}} \left\| \nabla e^{it\Delta_\Omega } \omega_n^J \right\|_{L^5 L^{\frac{30}{11}}} \\
         &+  \left\|  \psi_\varepsilon^j \right\|_{L^\infty_{t,x}} \left\|  \nabla e^{it \Delta_\Omega} \omega_n^J \right\|_{L^\frac{5}{2}L^{\frac{30}{17}}(\{|t| \leq T,\; |x| \leq R  \})} \\
          & \leq C \varepsilon + CR^{\frac{31}{60}} T^{\frac{1}{5}} \left\| e^{it\Delta_\Omega} \omega_n^J \right\|_{L^5_{t,x}}^{\frac{1}{6}} \left\| \omega_n^J \right\|^{\frac{5}{6}}_{H^1_0(\Omega)} .
    \end{align*}
    By taking the limit and choosing $\varepsilon$ small, we obtain \eqref{remains-term-local-smooth}. Hence, Claim \ref{the euror} holds.   
\end{proof}

Returning to the proof of the Proposition \ref{compac-prop}, we consider the other possibility.  \\
\underline{\bf Scenario \rm{II}:}  Only one nonzero profile. By \eqref{decomp of u_n}
\begin{equation*}
    u_n:=u(x,\tau_n)= \phi_n^1 + \omega_n^1,  
\end{equation*}
with \begin{equation}
  \label{w goes to 0 in H^1}  
\lim_{n \rightarrow \infty} \left\| \omega_n^1 \right\|_{H^1_0(\Omega)}=0.
\end{equation} 
If not, there exists $ \varepsilon>0$ such that $ \forall n,\; $ $$ E[\phi_n^1] M[\phi_n^1] \leq E_{\R^3}[Q] M_{\R^3}[Q]- \varepsilon, $$
and one can show by the previous argument that $u$ scatters in $H^1_0(\Omega).$ \\

It remains to show that $t_n^1$ is bounded and this will prove the convergence, up to a subsequence. \\ 
\begin{itemize}

\item If $t_n^1 \rightarrow + \infty$ (similarly, $t_n^1 \rightarrow -\infty$) and $\phi_n^1$ conforms to Case $1,$ i.e., $\phi_n^1= e^{it^1_n \Delta_\Omega} \phi^1,$
\begin{align*}
    \left\| e^{it\Delta_\Omega } u_n \right\|_{L^5_{t,x}([0,+\infty) \times \Omega)} &= \left\| e^{it\Delta_\Omega } \phi_n^1+ e^{it\Delta_\Omega } \omega_n^1 \right\|_{L^5_{t,x}([0,+\infty) \times \Omega)} \\ 
    &\leq   \left\| e^{i(t+t_n^1)\Delta_\Omega } \phi^1 \right\|_{L^5_{t,x}([0,+\infty) \times \Omega)} + \left\| \omega_n^1 \right\|_{H^1_0(\Omega)} \\
   &\leq \left\| e^{it\Delta_\Omega } \phi^1 \right\|_{L^5_{t,x}([t_n^1,+\infty) \times \Omega)} + \left\| \omega_n^1 \right\|_{H^1_0(\Omega)},
\end{align*}
which goes to $0$ as $n$ goes to $\infty$, showing that $u_n$ scatters for positive (similarly negative) time, a contradiction. 
\item If $t_n^1 \rightarrow + \infty$ (similarly, $t_n^1 \rightarrow -\infty$) and $\phi_n^1$ conforms to Case $2,$ i.e., $$\phi_n^1= e^{it_n^1 \Delta_{\Omega}}[(\chi_n^1 \phi^1)(x-x_n^1)], \quad \text{ where} \; \;   \chi_n^1:=\chi \left( \frac{x}{|x_n^1|} \right).$$
We first prove that
\begin{equation}
\label{conv-domain}
\lim_{n \rightarrow + \infty}   \left\| e^{i\, t \Delta_{\Omega_n} } (\chi_n^1 \phi^1 )- e^{i\, t \Delta_{\R^3} } (\chi_n^1 \phi^1 ) \right\|_{L^5_{t,x}( (0,+\infty) \times \R^3)}=0,
\end{equation}
where  $\Omega_n:=\Omega - \{x_n\}. $
Indeed,  by a density argument, for any $ \varepsilon >0, $ there exist $\psi_{\varepsilon}\in \mathcal{C}^{\infty}_c(\R^3)$ such that 
\begin{equation}
\label{density}
 \left\|  \phi^1 -\psi_{\varepsilon} \right\|_{H^1(\R^3)} \leq \frac{\varepsilon}{4}. 
\end{equation} 
By the definition of $\chi_n,$ as $|x_n| \longrightarrow + \infty,$ for any $\varepsilon >0$ there exists $ N_{\varepsilon} \in \N$ such that 
\begin{equation} 
\forall n \geq N_{\varepsilon}, \quad  \label{chiphi-phi}   \left\| \chi_n^1 \phi^1-\phi^1 \right\|_{H^1(\R^3)}  \leq \frac{\varepsilon}{4}. 
\end{equation}
Using \eqref{density} and  \eqref{chiphi-phi}, we have 
$$   \forall n \geq N_{\varepsilon}, \quad \left\| \chi_n^1  \phi^1 -\psi_{\varepsilon} \right\|_{H^1(\R^3)} \leq \frac{\varepsilon}{2}. $$
Combining this with the Strichartz inequality, we obtain for large $n$
\begin{align}
\left\| e^{ it \Delta_{\Omega_n}} \left( \chi^1_n \phi^1- \psi_{\varepsilon} \right)\right\|_{L^5_{t,x}((0,+\infty) \times \R^3)} + 
\left\| e^{ it \Delta_{\R^3}} \left( \chi^1_n \phi^1- \psi_{\varepsilon}  \right)\right\|_{L^5_{t,x}((0,+\infty) \times \R^3)}  \leq \frac{\varepsilon}{2}. 
\end{align}  
From \cite[Proposition 2.13]{KiVisnaZhang16}, as $|x_n|\longrightarrow + \infty, $ we have for large $n$
\begin{equation}
\label{conv-domain-comp}
\ \left\| e^{it \Delta_{\Omega_n}} \psi_{\varepsilon} - e^{it \Delta_{\R^3}} \psi_{\varepsilon} \right\|_{L^5_{t,x}((0,\infty) \times \R^3)} \leq  \frac{\varepsilon}{2},
\end{equation}
which yields \eqref{conv-domain}. We now have
\begin{align*}
    \left\| e^{it\Delta_\Omega } u_n \right\|_{L^5_{t,x}([0,+\infty) \times \Omega)} &= \left\| e^{it\Delta_\Omega } \phi_n^1+ e^{it\Delta_\Omega } \omega_n^1 \right\|_{L^5_{t,x}([0,+\infty) \times \Omega)} \\ 
    &\leq   \left\| e^{i(t+t_n^1)\Delta_\Omega } (\chi_n^1 \phi^1 )(x-x_n^1)\right\|_{L^5_{t,x}([0,+\infty) \times \Omega)} + \left\| \omega_n^1 \right\|_{H^1_0(\Omega)} \\
   &\leq \left\| e^{it\Delta_\Omega }  (\chi_n^1 \phi^1 )(x-x_n^1) \right\|_{L^5_{t,x}([t_n^1,+\infty) \times \Omega)} + \left\| \omega_n^1 \right\|_{H^1_0(\Omega)} \\
   &\leq   \left\| e^{i\, t \Delta_{\Omega_n} } (\chi_n^1 \phi^1 )- e^{i\, t \Delta_{\R^3} } (\chi_n^1 \phi^1 ) \right\|_{L^5_{t,x}( (t_n^1,+\infty) \times \R^3)}  \\ &+   \left\| e^{i\, t \Delta_{\R^3} } (\chi_n^1 \phi^1 ) \right\|_{L^5_{t,x}( (t_n^1,+\infty) \times \R^3)} + \left\| \omega_n^1 \right\|_{H^1_0(\Omega)} ,
\end{align*}
which goes to $0$ as $n$ goes to $\infty$, by \eqref{conv-domain} and the monotone convergence theorem, showing that $u_n$ scatters for positive (respectively, negative) time, a contradiction. This completes the proof of Proposition \ref{compac-prop}.
\end{itemize}
\end{proof}
\begin{corll}
\label{x=X}
Let $u$ be as in Proposition \ref{compac-prop}. Then one can choose the continuous function $x(t)$ such that $X(t)=x(t)$ for all $t \in D_{\delta_{0}},$ and the set $K$ has a compact closure in $H^1(\R^3).$
\end{corll}

\begin{proof}
Recall that by the definition of $D_{\delta_{0}}$, the modulation parameters $X(t), \theta(t)$ and $\alpha(t)$ are well defined for all $t \in D_{\delta_{0}}.$ Let $x(t)$ be the translation parameter given by Proposition \ref{compac-prop}. Let $R_0>0.$ Then by the decomposition of $u$ in \eqref{decom-u-trans}, Proposition \ref{lem-equiv-modu-param} and the fact $\Psi(x)=1$ for $|x|$ large, there exists ${C_{\star}}>0$ such that $$ \forall t \in D_{\delta_{0}} ,  \qquad \int_{|x|\leq R_0} \left| \nabla Q \right|^2+ |Q|^2 - C_{\star} \left( \delta(t) +\e \right) \leq \int_{|x-X(t)| \leq R_0}  \left| \nabla \uu \right|^2 + |\uu|^2  .$$
Taking $\delta_0$ small if necessary,  there exists $\varepsilon_0 >0$ such that 
$$ \forall t \in D_{\delta_{0}}, \qquad \int_{|x+x(t)-X(t)|\leq R_0} \left| \nabla \uu(t,x+x(t))\right|^2 + |\uu(t,x+x(t))|^2 \geq \varepsilon_0 >0. $$
Using the fact that $K$ has a compact closure in $H^1(\R^3),$ we get that $|x(t)-X(t)|$ is bounded. Thus, one can modify $x(t)$ such that  $\overline{K}$ remains compact and  for all $t$ in $D_{\delta_0},$ $x(t)=X(t).$  
\end{proof}
\subsection{Control of the translation parameters}
\label{sec-bddxx}
\begin{prop}
\label{bddX}
Consider a solution $u$  of \Nls such that 
\begin{equation}
\label{ME_u=ME_Q+Grad}
M[u]=M_{\R^3}[Q], \; E[u]=E_{\R^3}[Q], \; \left\| \nabla u_0 \right\|_{L^2(\Omega)} <  
\left\| \nabla Q \right\|_{L^2(\R^3)}
\end{equation}
and 
\begin{equation}
\label{compacUbar}
 K:=\{ \underline{u}(t,x+x(t)); t \geq 0 \} 
 \end{equation}
has a compact closure in $H^1(\R^3).$ Then $x(t)$ is bounded. 
\end{prop}

We start with the following lemma. 
\begin{lem}
\label{suiteDelta-->0}
Let $u$ be as in the Proposition \ref{bddX}. Let $\{t_n\}$ be a sequence of time, such that $t_n \longrightarrow + \infty.$ Then $ | x(t_n) | \longrightarrow + \infty$ as $n \rightarrow+\infty,$ if and only if   $\delta(t_n) \longrightarrow 0$ as $n $ goes to $+ \infty . $
\end{lem}

\begin{proof}
We first prove that $\delta(t_n) \longrightarrow 0$ implies that $|x( t_n ) | \longrightarrow+ \infty$ as $ n \rightarrow + \infty.$  If not, $x(t_n)$ converges (after extraction) to $x_\infty$ in $\R^3.$ By the compactness of the closure of $K,$ $\uu(t_n, \cdot+x(t_n))$ converges in $H^1(\R^3)$ to some $v_0(\cdot-x_\infty) \in H^1(\R^3).$ By the assumption \eqref{ME_u=ME_Q+Grad} and the fact that $\delta(t_n) \longrightarrow 0, \; E_{\R^3}[v_0]=E_{\R^3}[Q], M_{\R^3}[v_0]=M_{\R^3}(Q)$ and $\left\| \nabla v_0 \right\|_{L^2(\R^3)}= \left\| \nabla Q \right\|_{L^2(\R^3)}.$
By Proposition \ref{uConvQ}, there exist $\theta_0 \in \R$ and $ x_0 \in \R^3$ such that $v_0=e^{i \theta_0} Q(\cdot-x_0).$ On the other hand, if $x+x(t_n) \in \Omega,$ then $u(t_n,x+x(t_n) )$ converges in $H^1_0(\Omega),$ as $H^1_0(\Omega)$ is a close subspace of $H^1(\R^3).$ Thus, the restriction of $v_0(\cdot-x_\infty)$ to $\Omega$ belongs to $H^1_0(\Omega),$ which contradicts the fact that $e^{i \theta_0}Q(\cdot + x_{\infty}-x_0) \notin H^1_0(\Omega)$.    \\

Next, we prove that $ | x(t_n) | \longrightarrow + \infty$ as $n \rightarrow+\infty$ implies that $\delta(t_n) \longrightarrow 0$ as $n $ goes to $ + \infty . $ \\ 
We argue by contradiction, assuming (after extraction) that  $$\delta(t_n) \xrightarrow[ n \rightarrow+\infty] {}\delta_{\infty}>0 \; \text { and } \; t_n  \xrightarrow[ n \rightarrow+\infty] {} t_{\infty}\in \R \cup \{\pm \infty\}.$$
By the continuity of $x(t),$ using $|x(t_n)| \rightarrow + \infty,$ we must have $t_{\infty} \in \{\pm \infty\}.$ \\
Assume, say, $t_{\infty}=+ \infty,$  and let $\varphi_{\infty}= \displaystyle \lim_{n \rightarrow + \infty} \uu(t_n,x+x(t_n))$ in $H^1(\R^3)$ (after extraction). We have 
\begin{align*}
E_{\R^3}[\varphi_{\infty}]=E_{\R^3}[Q], \quad M_{\R^3}[\varphi_{\infty}]=M_{\R^3}[Q] , \quad  \int_{\R^3}  \left| \nabla \varphi_{\infty} \right|^2= \int_{\R^3} \left|   \nabla Q \right|^2 - \delta_{\infty} < \int_{\R^3} \left| \nabla Q \right|^2   .
\end{align*}
Let $\varphi$ be the solution of \nls  with the initial datum $\varphi_{\infty}$ at $t=0.$ By \cite{DuRo10}, $\varphi$ is global and one of the following holds:
\begin{enumerate}
\item $\varphi$ scatters in both time directions.  
\item $\exists \, \tau, \theta \in \R$ and $ \varepsilon \in \{ \pm 1 \}$  such that $\varphi(t)= e^{i \theta} U_{-}(\varepsilon \,  t + \tau) ,$ where $U_{-}(t) \xrightarrow[t \rightarrow + \infty ]{ }  Q $  and $U_{-}$~scatters for negative time.
\end{enumerate}

In case $(1)$ or in the case $(2)$ with $\varepsilon= -1,$ one can prove by approximation, following the proof of Theorem $4.1$ in \cite{KiVisnaZhang16}, that $u$ scatters for positive time. \\ 
In case $(2)$ with $\varepsilon=+1$, we obtain for large $n,$ with the same argument
$$ \left\| u \right\|_{S(-\infty,t_n)}  \leq C \;  \left\| U_{-} \right\|_{S(-\infty, t_{\infty})}, \quad \text{ where  $C$ is a fixed constant.}  $$ 
Letting $n$ go to $+ \infty,$ we see that $u$ has a finite Strichartz norm, thus, $u$ scatters also in both time directions, which contradicts the fact that $u$ satisfies \eqref{compacUbar} and \eqref{ME_u=ME_Q+Grad}.
\end{proof}
\begin{lem}
\label{lem-aproox-X-delta}
Let $X(t)$ be as in \eqref{decom-u}. Taking a smaller $\delta_0$ if necessary, there exists $C>0$ such that
\begin{equation}
\label{aproox-delta}
   \frac{ e^{- |X(t)|} }{\left|X(t) \right| } \leq C \delta(t)  \quad \mbox{for any } \;  t\in D_{\delta_0}.
\end{equation}
\end{lem}
\begin{proof}
Note that, by Proposition \ref{compac-prop}, taking a smaller $\delta_0$ if necessary, we can assume $|X(t)|\geq C$ for an arbitrarily large constant $C>0$. The proof consists of $3$ steps.
\begin{itemize}
\item \underline{\bf Step 1:} The estimate of $\delta(t)$ with respect to an auxiliary modulation parameter $X_1(t)$ on $\R^3.$ Let  $\uu(t) \in H^1(\R^3)$ be the extension of $u$ to $\mathbb{R}^3$ defined as in \eqref{defunderlineU}, we then have 
 \begin{equation}
M _{\R^3}[\uu]=M _{\R^3}[Q], \quad E_{\R^3}[\uu]=E _{\R^3}[Q], \quad \text{and} \quad \int_{\R^3} \left| \nabla \uu \right|^2 < \int_{\R^3} \left| \nabla Q \right|^2 .
\end{equation}
Arguing as in Section $3,$ but on the whole space $\R^3$, see \cite[Lemma 4.1 and 4.2]{DuRo10}, there exist  $\theta_1(t)$ and $X_1(t) ,$  $C^1$ functions of $t$, such that 
\begin{equation}
\label{decom-ubar}
 e^{-i\theta_1(t)-it}\uu(t,x+X_1(t))=(1+\rho_1(t))Q(x)+ \widetilde{h}(t,x),
\end{equation}
where
\begin{align}
  \rho_1(t) = \re \frac{ e^{-i \theta_1-it} \int_{\R^3} \nabla \uu(t,x+X_1(t)). \nabla Q(x) dx}{ \left\| \nabla  Q \right\|_{L^2(\R^3)}^2} -1 , \\ 
  \label{equiv-modu-ubar-para}
  \left| \rho_1(t) \right| \approx \left| \int_{\R^3} Q  \, \hh \, dx  \right| \approx \left\|  \hh \right\|_{H^1(\R^3)} \approx \delta(t) .
\end{align}
In this step we prove
\begin{equation}
\label{Q+X-1+delta}
     \frac{e^{- |X_1(t)|}}{\left| X_1(t) \right| } \leq C \delta(t).
\end{equation}

By \eqref{decom-ubar},  $x \in \Omega^c$ implies  $(1+\rho_1(t))Q(x-X_1(t))+ \hh(t,x-X_1(t))=0,$ i.e.,  
$$\left\| (1+\rho_1(t))Q(x-X_1(t))+ \hh(t,x-X_1(t)) \right\|_{L^2( \Omega^c)} =0. $$ 
By \eqref{equiv-modu-ubar-para}, we have 
\begin{equation}
\label{Q-comp-Omega-with-delta}
  \int_{\Omega^c} | Q(x-X_1(t))|^2 \, dx  \leq C \, \delta(t)^2.
  \end{equation}
By \eqref{x0large}, one can see that $\left| X_1(t) \right|$ is large. For $x \in \Omega^c,$ we have $$ \frac 12 |X_1(t)|   \leq|x-X_1(t)| \leq 2 |X_1(t)|.$$ 
From Lemma \ref{decay1 Q}, we have 
$$ Q(x)= \frac{e^{-|x|}}{|x|} \bigg( a+ O(\frac{1}{|x|^{\frac 12}}) \bigg), \quad \text{for some}  \; a>0.$$
Using  \eqref{Q-comp-Omega-with-delta}, we obtain  \eqref{Q+X-1+delta}.  \\

\item \underline{\bf Step 2:} Comparison of $X(t)$ and $X_1(t). $\\

We prove that there exists $C>0$ such that 
\begin{equation}
\label{compa-X-X_1}
  \left| X(t)-X_1(t)\right| \leq C \quad  \forall t \in D_{\delta_0}.
\end{equation}
We fix $t\in D_{\delta_0}.$ We can assume 
\begin{equation}
\label{X(t)-X1(t)grand}
|X(t)-X_1(t)| \geq 1,
\end{equation} 
or else we are done.

Let $x\in \Omega,$ by \eqref{decom-ubar} and \eqref{decom-u-trans}, we have 
\begin{align*}
    u(t,x)&=e^{i\theta(t)+it} (1+\rho (t))Q(x-X(t)) \Psi(x)+ e^{i\theta(t)+it} h(t, x)\\
    &=e^{i\theta_1(t)+it} (1+\rho_1 (t))Q(x-X_1(t)) + e^{i\theta_1(t)+it}\,  \hh(t, x).
\end{align*}
Using \eqref{equiv-modu-ubar-para} and Proposition \ref{lem-equiv-modu-param} , we have 
\begin{align*}
  \int_{|x-X(t)|<1}  \left|Q(x-X(t)) \Psi(x)e^{i \theta(t)} - Q(x-X_1(t))e^{i\theta_1(t)}\right|^2 \leq C \left (\delta^2(t)+ \frac{e^{-2|X(t)|}}{|X(t)|^2}  \right).
\end{align*}
Recall that $|X_1(t)|$ and $|X(t)|$ are large and $\Psi(x)=1$ for large $|x|.$ \\ 
\begin{align*}
    \int_{|x|<1} |Q(x)|^2 dx &\leq C \int_{|x-X(t)|<1} |Q(x-X_1(t))|^2 dx + C \delta^2(t)+ C \,  \frac{e^{-2|X(t)|}}{|X(t)|^2} 
    \\ & \leq \int_{|x-X(t)|<1} \frac{e^{-2|x-X_1(t)|}}{|x-X_1(t)|^2} \, dx + C \delta^2(t)+ C \, \frac{e^{-2|X(t)|}}{|X(t)|^2} .
\end{align*}
Using the fact that $|x-X_1(t)|\geq |X(t)-X_1(t)|-|x-X(t)|\geq   |X(t)-X_1(t)|-1,$ in the support of the integral in the last line, we obtain
\begin{align*}
     \int_{|x|<1} |Q(x)|^2 dx  \leq C  \frac{e^{-2|X(t)-X_1(t)|}}{|X(t)-X_1(t)|^2} + C \, \delta^2(t) + C \, \frac{e^{-2|X(t)|}}{|X(t)|^2}  .
\end{align*}
Recall that, by Lemma \ref{suiteDelta-->0} if $|X(t)| $ is large, then $\delta(t)$ and $\, \frac{e^{-2|X(t)|}}{|X(t)|^2} $ are small. By \eqref{X(t)-X1(t)grand}, we get 
$$\frac 12 \int_{|x|<1} |Q(x)|^2 \, dx \leq C  \;  \frac{e^{-2|X(t)-X_1(t)|}}{|X(t)-X_1(t)|^2} \leq C e^{-2|X(t)-X_1(t)|},$$
which yields
\begin{align*}
    |X(t)-X_1(t)| \leq C - \log\left( \frac 12 \int_{|x|<1} |Q(x)|^2 \, dx \right) .
\end{align*}
Thus, $|X(t)-X_1(t)|$ is bounded.\\

\item \underline{\bf Step 3:} Conclusion of the proof. \\ 
From Step 2 we have $ |X(t)-X_1(t)|\leq C,$ and since $|X(t)|$ is large, we have   
\begin{equation}
\frac 12 |X(t)|  \leq |X(t)|-|X(t)-X_1(t)|    \leq |X_1(t)|\leq |X_1(t)-X(t)|+|X(t)|\leq 2 |X(t)| . 
\end{equation}
By Step $1$, we get $ \delta^2(t) \geq C \, \frac{e^{-2|X_1(t)|}}{|X_1(t)|^2} ,$ which implies 
$$\delta^2(t) \geq C \, \frac{ e^{-2|X(t)|}}{|X(t)|^2},$$
concluding the proof of Lemma \ref{lem-aproox-X-delta}.
\end{itemize}
\end{proof}

\begin{lem}
\label{intdelta}
Let $u$ be a solution of \Nls satisfying the assumptions of  the Proposition \ref{bddX}. Then there exists a constant $C>0$ such that  if $0 \leq \sigma \leq \tau$ 
\begin{equation}
\label{integ-delta}
\int_{\sigma}^{\tau} \delta(t)  \leq C  \left[ 1+ \sup_{ t \in [\sigma, \tau]} |x(t)| \right] \left( \delta(\sigma)+ \delta(\tau) \right)   .
\end{equation}
\end{lem}
\begin{proof}
Let $\varphi$ be a smooth radial function such that
\begin{equation*}
    \varphi(x):=\begin{cases}
    |x|^2 \; \text{ if } &|x| \;  \leq 1, \\
    0 \;  \quad \text{ if } &|x| \geq 2.
    \end{cases}
\end{equation*}
Consider the localized variance,   
\begin{equation}
\label{def-Y-R}
 \mathcal{Y}_R(t)= \int_{\Omega} R^2 \varphi\left(\frac xR \right) |u(t,x)|^2 \, dx,
 \end{equation}
where $R$ is large positive constant, to be specified later. 
Then, 
\begin{equation}
\label{derv1-Y-R}
 \mathcal{Y}'_{R}(t)=2R\im \intom \bar{u} \, \nabla \varphi\left(\frac{x}{R}\right) \cdot \nabla u \, dx ,\quad  \; \left| \mathcal{Y}'_{R}(t) \right| \leq C\, R .
 \end{equation}
Furthermore, $$  \mathcal{Y}''_{R}(t)= 8 \int_{\Omega} \left| \nabla u \right|^2 \, dx-6 \int_{\Omega} |u|^4+ A_R(u(t)) \, dx -2  \intb \left| \nabla u \right|^2 x \cdot \nn \, d\sigma(x),$$
where $\nn$ is the outward normal vector and 
  \begin{multline}
\label{def-A_R}
A_R(u(t)):= 4 \sum_{j \neq k }\intom \frac{\partial^2 \varphi}{\partial x_j \partial x_k} \left(\frac xR \right) \frac{ \partial u }{\partial x_j} \frac{\partial \bar{u} }{\partial x_k} + 4 \sum_{j} \intom \left( \frac{\partial^2 \varphi}{\partial x_j^2 } \left(\frac xR \right) -2\right) \left| \partial_{x_j} u \right|^2\\   - \frac{1}{R^2} \intom |u|^2 \Delta^2\varphi \left(\frac xR \right) - \intom \left( \Delta \varphi\left(\frac xR \right)-6\right) |u|^4.  \qquad \qquad \qquad \qquad \qquad
\end{multline}
As $\partial \Omega$ is convex and $0\in \Omega,$ one can see that $x \cdot \nn \leq 0 ,$ for all $x\in \partial \Omega.$ Thus, 
$$-2 \int_{\partial \Omega}  \left| \nabla u \right|^2 \, x \cdot \nn \;d\sigma(x)=2\int_{\partial \Omega}\left| \nabla u \right|^2 \, |x \cdot \nn| \; d\sigma(x). $$
Using the fact $\left\|Q\right\|_{L^4}^4=\frac 43 \left\| \nabla u \right\|_{L^2}^2$ and $E[u]=E_{\R^3}[Q],$ we have 
$ 8 \left\| \nabla u \right\|_{L^2}^2-6 \left\| u \right\|_{L^4}^4=4 \delta(t) ,$ 
which yields 
\begin{equation}
\label{deriv2-Y}
 \mathcal{Y}''_{R}(t)= 4 \delta(t)+ A_R(u(t))+ 2  \intb \left| \nabla u \right|^2 |x \cdot \nn| \; d\sigma(x).
\end{equation}
\begin{itemize}
    \item \underline{\bf Step 1:} Bound on $A_R.$ \\
    
 In this step we prove: for $\varepsilon >0,$ there exists a constant $R_\varepsilon >0$ such that
\begin{equation}
    \label{bdd-A_R}
    \forall t \geq 0 , \; R \geq R_{\varepsilon} ( 1+ |x(t)|) \Longrightarrow \left| A_R(u(t))\right| \leq \varepsilon \delta(t). 
\end{equation}

We distinguish two cases: $\delta$ small or not.
In the first case, we will use the estimate on the modulation parameters in Section \ref{Sec-Modulation}. Consider $\delta_0>0,$ as in the previous Section, such that the modulation parameters, 
$\Theta(t),X(t),\rho(t)$ are well defined for all $t\in D_{\delta_0}.$ Let $\delta_1$ to be specified later such that $0<\delta_1<\delta_0.$ Assume that $t \in D_{\delta_1}.$ Let $g_{_{-X}}=\rho Q_{_{-X}} \Psi+ h,$ then from Proposition \ref{lem-equiv-modu-param} with Lemma \ref{lem-aproox-X-delta} and \eqref{decom-u}, we have 
\begin{equation}
    \label{decom-u+g}
    u(t,x)=e^{i\theta(t)+it } Q(x-X(t))\Psi(x)+ g(t,x-X(t)) e^{i\theta(t)+it }\quad \text{ and } \left\| g \right\|_{H^1_0(\Omega)} \leq C  \delta(t)  .
\end{equation}

For the equation on $\R^3,$ $\mathcal{Y}_R$ is defined by integration on $\R^3$ (not on $\Omega$). The corresponding $A_R$ is also defined with an integration on $\R^3$ instead of $\Omega.$ A crucial point is that when $R$ is large, by the property of $ \varphi,$ all the integrands in the definition \eqref{def-A_R} of $A_R$ are $0$ close to the obstacle, so that it can be integrated over $\R^3$ instead of $\Omega.$ Note that it only works if $R$ is large enough. If $\theta_0$ and $x_0$ are fixed, $e^{i\theta_0 + it }Q(\cdot+x_0)$ is a solution of \nls such that the corresponding $\mathcal{Y}_R(t)$ does not depend on $t$ and also $\delta(t)=0.$ Thus, $A_R(e^{i{\theta_0}+it}Q(\cdot+x_0))=0 $ for all $R$ and $t.$  \\

Using the change of variable $y=x-X(t)$ in \eqref{def-A_R},  we get
\begin{align*}
\left| A_R(u(t)) \right|&= \left| A_R(u(t))-A_R(e^{i\theta(t)+it}Q(x-X(t))   \right| \\& \leq C \int_{|y+X(t)|\geq R }  \bigg(|\nabla Q(y) | |\nabla g (y)|+|\nabla g (y)|^2+ |Q(y)| |g(y)|+ |Q(y)||g(y)|^3 \\ &\qquad \qquad \qquad \qquad + |g(y)|^2+ |g(y)|^4 \bigg) dy   
\\ & \leq C \int_{|y+X(t)|\geq R } \bigg( \frac{e^{-|y|}}{|y|} \left(|\nabla g (y)|+|g(y)|+|g(y)|^3  \right) + |\nabla g (y)|^2+|g(y)|^2+|g(y)|^4 \bigg)  dy. 
\end{align*}
By \eqref{decom-u+g}, we have $\left\| g \right\|_{H^1_0(\Omega)} \leq C \delta(t) , $ which yields  
\begin{align*}
    R \geq R_0+ |X(t)| \Longrightarrow \left| A_R(u(t)) \right|  &\leq C \left[ e^{-R_0}(\delta(t)+\delta(t)^3)+ \delta(t)^2+\delta(t)^4  \right] \\ &\leq C \left[ e^{-R_0}+e^{-R_0}\delta(t)^2 +\delta(t)+\delta(t)^3 \right]\delta(t) \\ 
    &\leq \varepsilon \delta(t),
\end{align*}
provided $R_0>0$ is such that $Ce^{-R_0}\leq \frac{\varepsilon}{2} $ and $\delta_1$ is such that $Ce^{-R_0}\delta_1^2 +\delta_1+\delta_1^3 \leq \frac \varepsilon 2.$ \\ 
Since $0<\delta_1<\delta_0$ and $x(t)=X(t)$ on $D_{\delta_0},$
we obtain \eqref{bdd-A_R} for $\delta(t)<\delta_1.$  \\ 
Now consider the second case, i.e., $\delta(t)\geq \delta_1. $ By \eqref{def-A_R}, we have 
\begin{align*}
    \left| A_R(u(t))\right| \leq C \int_{|x-x(t)|\geq R-|x(t)|} | \nabla u (t)|^2+ |u(t)|^4 + |u(t)|^2 dx .
\end{align*}
By the compactness of $K,$ there exists $R_1>0$ such that 
\begin{equation}
    R \geq |x(t)|+R_1 \text{ and } \delta(t)\geq \delta_1 \Longrightarrow \left| A_R(u(t)) \right| \leq  \varepsilon \delta_1 \leq \varepsilon \delta(t),
\end{equation}
which concludes the proof of \eqref{bdd-A_R} and completes Step 1. \\ 

\item \underline{\bf Step 2:} Conclusion of the proof. \\ 
By \eqref{deriv2-Y} and \eqref{bdd-A_R}, we get that there exists $R_2>0$ such that, 
\begin{align*}
    R \geq R_2(1+|x(t)|) \Longrightarrow \left| \mathcal{Y}^{''}_{R}(t) \right| \geq 2 \delta(t).
\end{align*}
Let $R=R_2(1+\sup_{\sigma \leq t \leq \tau } |x(t)|).$ Then
\begin{align}
\label{delta<y2<y1}
    2 \int_{\sigma}^{\tau} \delta(t) dt \leq \int_{\sigma}^{\tau} \mathcal{Y}^{''}_R(t) \,  dt \leq \mathcal{Y}^{'}_R(\tau)-\mathcal{Y}^{'}_R(\sigma). 
\end{align}
If $\delta(t)<\delta_0,$ then by Step 1, changing the variable $y=x-X(t)$ and since $\Psi(x)=1$ for large $|x|,$ we obtain 
\begin{align*}
    \mathcal{Y}^{'}_R(t)&=2 R \im \int \bar{g}(y)\; \nabla \varphi\left(\frac{y+X(t)}{R}\right) \cdot \nabla \left(Q(y) \Psi(y+X(t)\right)\\ & + 2 R \im \int Q(y) \Psi(y+X(t)) \nabla\varphi \left( \frac{y+X(t)}{R} \right) \cdot \nabla g(y) \, dy \\ 
   & + 2R \im \int \bar{g}(y) \nabla\varphi \left( \frac{y+X(t)}{R} \right) \cdot  \nabla g(y) dy ,
\end{align*}
which yields 
$$  \left| \mathcal{Y}^{'}_R(t) \right| \leq CR(\delta(t)+\delta(t)^2) \leq CR \delta(t). $$
This inequality is also valid for $\delta(t) \geq \delta_0,$ by straightforward estimates. Using \eqref{delta<y2<y1}, we obtain 
\begin{align*}
\int_{\sigma}^{\tau} \delta(t)dt &\leq C\, R (\delta (\sigma)+ \delta(\tau)) \\ &\leq C \, R_2 \left(1+\sup_{\sigma \leq t \leq \tau } |x(\tau)| \right) (\delta (\sigma)+ \delta(\tau)).
\end{align*}
This concludes the proof of Lemma \ref{intdelta}.
\end{itemize}

\end{proof}
\begin{lem}
\label{x(tau)-x(sigma)}
There exists a constant $C>0$ such that 
\begin{equation}
\label{eq-x(tau)-x(sigma)}
\forall \sigma,\tau>0\quad \text{with} \quad \sigma+1 \leq \tau, \quad |x(\tau)-x(\sigma)| \leq C \int_{\sigma}^{\tau} \delta(t) dt  .
\end{equation}
\end{lem}
\begin{proof}
Let $\delta_0>0$ be as in Section \ref{Sec-Modulation}. Let us first show that there exists $\delta_1>0$ such that, 
\begin{equation}
\label{inf-supp}
    \forall \tau \geq 0 \quad \inf_{t\in[\tau,\tau+2]} \delta(t) \geq \delta_1 \quad \text{ or } \quad \sup_{t\in[\tau,\tau+2]} \delta(t) < \delta_0.
\end{equation}
If not, there exist $t_n,t_n'\geq 0$ such that 
\begin{equation}
\label{contradInf+sup}
    \delta(t_n) \xrightarrow[n \rightarrow +\infty]{} 0 ,\quad  \delta(t_n')\geq \delta_0, \quad |t_n-t_n'|\leq 2  ,
\end{equation}
extracting a subsequence if necessary, we may assume 
\begin{equation}
\label{def-tau}
    \lim_{n \rightarrow +\infty} t_n-t_n'=\tau \in [-2,2].
\end{equation}
Note that if $t'_n$ goes to $+ \infty,$ then $\left| x(t_n')\right| $ converges (after extraction) to a limit $X_0\in \R^3.$ If not $ \left| x(t_n') \right|  \longrightarrow + \infty $ and by Lemma \ref{suiteDelta-->0}, $\delta (t'_n) \longrightarrow 0,$ which contradicts \eqref{contradInf+sup}. \\
By the compactness of $K,$ we have  \begin{equation*}
\uu(t_n', \cdot + x(t_n'))  \xrightarrow[n \longrightarrow + \infty ]{} w_0  \in H^1(\R^3).
\end{equation*}
Denote $v_0(x)=w_0(x-X_0).$ We have 
\begin{equation}
\uu(t_n', \cdot+x(t_n^{'})) \xrightarrow[n \longrightarrow + \infty ]{}    v_0(\cdot+X_0)     \in H^1(\R^3).
\end{equation}
Thus,
$$\uu(t_n') \xrightarrow[n \longrightarrow + \infty ]{}  v_0 \in  H^1(\R^3). $$
In particular, $v_0=0$ on $\Omega^c$ and we obtain, 
\begin{equation} 
\label{conpac-conv}
u(t_n')   \xrightarrow[n \longrightarrow + \infty ]{}   v_0    \in    H^1_0(\Omega).
\end{equation}
Since $\delta(t_n')= \int \left| \nabla Q \right|^2 - \int \left|\nabla \uu(t_n',\cdot+x(t_n'))\right|^2 \geq \delta_0 >0,$ we have
\begin{equation*}
    \left\|  \nabla v_0 \right\|_{L^2(\Omega)}  < \left\| \nabla Q \right\|_{L^2(\R^3)} .
\end{equation*}
Let $v(t)$ be a solution of \Nls with initial data $v_0$ at $t=0$ and maximal time of existence~$I.$ Then by continuity of the flow of the \NNls equation, we have for all $t \in I,$
\begin{equation}
\label{v(t)-conitui-flow}
    \left\| \nabla  v(t) \right\|_{L^2(\Omega)}  < \left\| \nabla Q \right\|_{L^2(\R^3)} .
\end{equation}
As a consequence, $I=\R$ and by continuity of the flow of the \NNls equation, \eqref{def-tau} and \eqref{conpac-conv}, we have  
    $$u(t_n) \xrightarrow[n \longrightarrow + \infty]{} v(\tau) \in H^1_0(\Omega).$$ 
 Since $\delta(t_n) \rightarrow 0, \;  \left\| \nabla v(\tau) \right\|_{L^2(\Omega)} = \left\| \nabla Q  \right\|_{L^2(\R^3)}, $ which contradicts \eqref{v(t)-conitui-flow}. \\

Now, we prove \eqref{eq-x(tau)-x(sigma)} with an additional condition that $\tau <\sigma+2.$ By \eqref{inf-supp}, 
we may assume that  $$  \inf_{t\in[\sigma,\tau]} \delta(t) \geq \delta_1 \quad \text{ or } \quad \sup_{t\in[\sigma,\tau]} \delta(t) < \delta_0.$$
In the first case, we have $\int_{\sigma}^{\tau} \delta(t) \geq \delta_1$ and by a straightforward consequence of the compactness of $K$ and the continuity of the flow of \Nls equation, we have 

$$ \exists C>0,  \; \forall t,s \geq 0,\quad |t-s|\leq 2 \Longrightarrow  \left| X(t)-X(s) \right| \leq \frac{C}{\delta_1} \int_{\sigma}^\tau \delta(t) dt. $$
In the second case, by Corollary \ref{x=X} we have, $\forall t \in D_{\delta_0}, \; x(t)=X(t),$ and from Lemmas \ref{derv-of-mod-par} and \ref{lem-aproox-X-delta}, we have \begin{equation}
\label{X'+delta}
\left| X'(t) \right|  \leq C \delta(t) .
\end{equation}
 Thus, \eqref{eq-x(tau)-x(sigma)} follows from the time integration of \eqref{X'+delta} for $\tau < \sigma+2.$\\
 
 To conclude the proof of Lemma \ref{x(tau)-x(sigma)}, we divide $[\sigma,\tau]$ into intervals of length at least $1$ and at most $2$ and combine together the previous inequalities to get \eqref{eq-x(tau)-x(sigma)}.
\end{proof}
\begin{proof}[Proof of the Proposition \ref{bddX}]
We argue by contradiction. Assume that there exists $\tau_n \longrightarrow + \infty $ such that $ |x(\tau_n)| \longrightarrow + \infty$ and $|x(\tau_n)|=\sup_{ t \in [0,\tau_n]} |x(t)|. $ By Lemma \ref{suiteDelta-->0},  $\delta (\tau_n) \xrightarrow[ n \rightarrow + \infty ]{} 0.$

 Let $N_0$ be such that  $C\delta(\tau_n)\leq \frac{1}{100}$ for all $n\geq N_0.$ By Lemmas \ref{intdelta} and \ref{x(tau)-x(sigma)} we have \begin{align*}
 |x(\tau_n)-x(\tau_{N_0})| &\leq C \int_{\tau_{N_0}}^{\tau_n} \delta (t) dt \\ 
 & \leq C ( 1+|x(\tau_n)|) ( \delta(\tau_{N_0}) + \delta(\tau_{n})) ,
 \end{align*}
 hence, $$  |x(\tau_n)|  \leq C |x(\tau_{N_0})| ,$$
 which gives a contradiction. This concludes the proof of Proposition \ref{bddX}.
\end{proof}
\subsection{Convergence in mean}
\label{Sec-convergence-mean}
\begin{lem}
\label{lem-convinmean}
Consider a solution $u(t)$  of \Nls satisfying assumptions of  Proposition \ref{bddX}. Then 
\begin{equation}
    \label{cvinmean}
        \lim_{T \rightarrow + \infty} \frac 1T \int_{0}^T \delta(t) dt=0.
    \end{equation}
\end{lem}
\begin{corll}
Under the assumptions of Proposition \ref{bddX}, there exists a sequence of times $t_n$  such that $t_n \rightarrow + \infty$ and 
$$ \lim_{n \longrightarrow + \infty} \delta(t_n)=0.$$
\end{corll}
\begin{proof}[Proof of Lemma \ref{lem-convinmean}]

Consider the localized variance defined in \eqref{def-Y-R} and recall that from the proof of Lemma \ref{intdelta}, we have 
 \begin{equation}
\label{derv2-y}
 \mathcal{Y}''_{R}(t)= 4 \delta(t)+ A_R(u(t))+ 2  \intb \left| \nabla u \right|^2 |x \cdot \nn| d\sigma(x),
\end{equation}
where $\nn$ is outward normal vector and $A_R$ is defined in \eqref{def-A_R}. \\

If $|y|\leq 1, (\Delta^2 \varphi)(y)=0, \partial_{x_j}^2 \varphi(y)=2$ and $\Delta \varphi(y)=6.$ Thus, 
\begin{equation}
\label{A_Rforx>R}
  \left| A_R(u(t)) \right| \leq C \int_{|x|\geq R} \left| \nabla u \right|^2+ |u|^4+ \frac{1}{R^2} |u|^2.
\end{equation}
Let $x(t)$ be as in Corollary \ref{x=X} and $K$ be defined by \eqref{def-K}. Let $\varepsilon >0.$ By the compactness of $K$ and Proposition \ref{bddX}, there exists $R_0(\varepsilon)>0$ such that 
\begin{equation}
    \forall t \geq 0, \quad \int_{|x-X(t)| \geq R_0(\varepsilon)} |\nabla u|^2+ |u|^2+|u|^4 \leq \varepsilon.
\end{equation}
Furthermore, $x(t)$ is bounded, and thus, $\frac{x(t)}{t} \xrightarrow[ t \rightarrow + \infty]{} 0. $ There exists $t_0(\varepsilon)$ such that 
$$\forall t \geq t_0(\varepsilon), \; |x(t)| \leq \varepsilon t .$$
Let $$T\geq t_0(\varepsilon), \;  R=\varepsilon T+ R_0(\varepsilon)+1 \quad  \text{for }  \,  t \in [t_0(\varepsilon),T].$$
Next, we use the fact that $|x(t)|\leq \varepsilon T $ and $R_0(\varepsilon)+\varepsilon T \leq R, $ to get 
\begin{multline}
\label{|x|<R}
    \int_{|x|\geq R} |\nabla u |^2 + |u|^4+\frac{1}{R^2} |u|^2 \leq \int_{|x-x(t)|+|x(t)| \geq R}  |\nabla u |^2 + |u|^4+\frac{1}{R^2} |u|^2
    \\ \leq \int_{|x-x(t)| \geq R_0(\varepsilon)}  |\nabla u |^2 + |u|^4+\frac{1}{R^2} |u|^2 \leq \varepsilon.
\end{multline}
By \eqref{derv1-Y-R}, we have
$$  \int_{t_0(\varepsilon)}^T \mathcal{Y}^{''}_{R}(t) dt \leq \left| \mathcal{Y}^{'}_{R}(T) \right|+\left| \mathcal{Y}^{'}_{R}(t_0(\varepsilon)) \right| \leq C\, R .$$

From \eqref{derv2-y}, \eqref{A_Rforx>R} and \eqref{|x|<R} we have
\begin{equation*}
    \int_{t_0(\varepsilon)}^T \delta(t) dt \leq C(R+T\varepsilon) \leq C  R_0(\varepsilon)+\varepsilon T+1,
\end{equation*}
where $C>0,$ independent of $T$ and $\varepsilon.$ \\
This yields
\begin{equation*}
    \frac 1T \int_{0}^T \delta(t) dt \leq \frac 1T \int_{0}^{t_0(\varepsilon)} \delta(t) \, dt + \frac CT (R_0(\varepsilon)+1)+ C\varepsilon.
\end{equation*}
Taking first limsup as $T \rightarrow + \infty,$ and letting $\varepsilon$ tend to $0,$ we obtain \eqref{cvinmean}.
\end{proof}

\begin{prop}
\label{porp-u=0}
Let $u$ be a solution of \Nls such that
\begin{equation}
M[u]=M_{\R^3}[Q], \; E[u]=E_{\R^3}[Q], \; \left\| \nabla u_0 \right\|_{L^2(\Omega)} <  \left\| \nabla Q \right\|_{L^2(\R^3)}
\end{equation}
and $K=\{ u(t); t \geq 0 \}$ has a compact closure in $H^1_0(\Omega).$ Then $u \equiv0.$
\end{prop}
\begin{proof}
If not, there exists a solution $u\neq0$ such that the assumptions of this Proposition are satisfied. From Lemma \ref{lem-convinmean},
 there exists $t_n$ such that $t_n \longrightarrow +\infty$ and $\delta(t_n)$ tends to $0.$ By the compactness of the closure of $K, \;  u(t_n) $ converges in $H^1_0(\Omega)$ to some $v_0 \in H^1_0(\Omega)$ and the fact that $\delta(t_n) $ tends to $0$ implies that $E[v_0]=E_{\R^3}[Q], M[v_0]=M_{\R^3}[Q]$ and $ \left\| \nabla v_0 \right\|_{L^2(\Omega)}= \left\| \nabla Q \right\|_{L^2(\R^3).}$ Thus, $v_0=e^{i\theta_0}Q(x-x_0) \notin H^1_0(\Omega),$ for some parameters $\theta_0 \in \R$ and $x_0 \in \R^3,$ which contradicts the fact that $v_0 \in H^1_0(\Omega).$ 
\end{proof}

\bibliographystyle{acm}
\bibliography{references.bib}

\begin{thebibliography}{10}

\bibitem{AnRa08}
{\sc Anton, R.}
\newblock Global existence for defocusing cubic {NLS} and {G}ross-{P}itaevskii
  equations in three dimensional exterior domains.
\newblock {\em J. Math. Pures Appl. (9) 89}, 4 (2008), 335--354.

\bibitem{BeLi83a}
{\sc Berestycki, H., and Lions, P.-L.}
\newblock Nonlinear scalar field equations. {I}. {E}xistence of a ground state.
\newblock {\em Arch. Rational Mech. Anal. 82}, 4 (1983), 313--345.

\bibitem{BlairSmithSogge2012}
{\sc Blair, M.~D., Smith, H.~F., and Sogge, C.~D.}
\newblock Strichartz estimates and the nonlinear {S}chr\"{o}dinger equation on
  manifolds with boundary.
\newblock {\em Math. Ann. 354}, 4 (2012), 1397--1430.

\bibitem{BuGeTz04a}
{\sc Burq, N., G{\'e}rard, P., and Tzvetkov, N.}
\newblock On nonlinear {S}chr\"odinger equations in exterior domains.
\newblock {\em Ann. Inst. H. Poincar\'e Anal. Non Lin\'eaire 21}, 3 (2004),
  295--318.

\bibitem{Cazenave03BO}
{\sc Cazenave, T.}
\newblock {\em Semilinear {S}chr\"odinger equations}, vol.~10 of {\em Courant
  Lecture Notes in Mathematics}.
\newblock New York University Courant Institute of Mathematical Sciences, New
  York, 2003.

\bibitem{CooperStrauss76}
{\sc Cooper, J., and Strauss, W.~A.}
\newblock Energy boundedness and decay of waves reflecting off a moving
  obstacle.
\newblock {\em Indiana Univ. Math. J. 25}, 7 (1976), 671--690.

\bibitem{DuHoRo08}
{\sc Duyckaerts, T., Holmer, J., and Roudenko, S.}
\newblock Scattering for the non-radial 3{D} cubic nonlinear {S}chr\"odinger
  equation.
\newblock {\em Math. Res. Lett. 15}, 6 (2008), 1233--1250.

\bibitem{OuKaiSvetTh20}
{\sc Duyckaerts, T., Landoulsi, O., Roudenko, S., and Yang, K.}
\newblock Interaction between solitons and obstacle in the nonlinear
  {S}chr\"odinger equation.
\newblock {\em preprint\/} (2020).

\bibitem{DuMe09a}
{\sc Duyckaerts, T., and Merle, F.}
\newblock Dynamic of threshold solutions for energy-critical {NLS}.
\newblock {\em Geom. Funct. Anal. 18}, 6 (2009), 1787--1840.

\bibitem{DuRo10}
{\sc Duyckaerts, T., and Roudenko, S.}
\newblock Threshold solutions for the focusing 3d cubic {S}chr\"odinger
  equation.
\newblock {\em Rev. Mat. Iberoam. 26}, 1 (2010), 1--56.

\bibitem{DuRo15}
{\sc Duyckaerts, T., and Roudenko, S.}
\newblock Going beyond the threshold: scattering and blow-up in the focusing
  {NLS} equation.
\newblock {\em Comm. Math. Phys. 334}, 3 (2015), 1573--1615.

\bibitem{FaXiCa11}
{\sc Fang, D., Xie, J., and Cazenave, T.}
\newblock Scattering for the focusing energy-subcritical nonlinear
  {S}chr\"odinger equation.
\newblock {\em Sci. China Math. 54}, 10 (2011), 2037--2062.

\bibitem{Gerard98}
{\sc G\'{e}rard, P.}
\newblock Description du d\'{e}faut de compacit\'{e} de l'injection de
  {S}obolev.
\newblock {\em ESAIM Control Optim. Calc. Var. 3\/} (1998), 213--233.

\bibitem{GiNiNi81}
{\sc Gidas, B., Ni, W.~M., and Nirenberg, L.}
\newblock Symmetry of positive solutions of nonlinear elliptic equations in
  {${\bf R}^{n}$}.
\newblock In {\em Mathematical analysis and applications, {P}art {A}}, vol.~7
  of {\em Adv. in Math. Suppl. Stud.} Academic Press, New York, 1981,
  pp.~369--402.

\bibitem{Guevara14}
{\sc Guevara, C.~D.}
\newblock Global behavior of finite energy solutions to the {$d$}-dimensional
  focusing nonlinear {S}chr\"{o}dinger equation.
\newblock {\em Appl. Math. Res. Express. AMRX}, 2 (2014), 177--243.

\bibitem{HoRo08}
{\sc Holmer, J., and Roudenko, S.}
\newblock A sharp condition for scattering of the radial 3{D} cubic nonlinear
  {S}chr\"odinger equation.
\newblock {\em Comm. Math. Phys. 282}, 2 (2008), 435--467.

\bibitem{Ivanovici07}
{\sc Ivanovici, O.}
\newblock Precised smoothing effect in the exterior of balls.
\newblock {\em Asymptot. Anal. 53}, 4 (2007), 189--208.

\bibitem{Ivanovici10}
{\sc Ivanovici, O.}
\newblock On the {S}chr{\"o}dinger equation outside strictly convex obstacles.
\newblock {\em Analysis \& PDE 3}, 3 (2010), 261--293.

\bibitem{MR2683754}
{\sc Ivanovici, O., and Planchon, F.}
\newblock On the energy critical {S}chr\"odinger equation in {$3D$}
  non-trapping domains.
\newblock {\em Ann. Inst. H. Poincar\'e Anal. Non Lin\'eaire 27}, 5 (2010),
  1153--1177.

\bibitem{Ivri69}
{\sc Ivri\u{\i}, V.~J.}
\newblock Exponential decay of the solution of the wave equation outside an
  almost star-shaped region.
\newblock {\em Dokl. Akad. Nauk SSSR 189\/} (1969), 938--940.

\bibitem{KeMe06}
{\sc Kenig, C.~E., and Merle, F.}
\newblock Global well-posedness, scattering and blow-up for the
  energy-critical, focusing, non-linear {S}chr\"odinger equation in the radial
  case.
\newblock {\em Invent. Math. 166}, 3 (2006), 645--675.

\bibitem{Keraani01}
{\sc Keraani, S.}
\newblock On the defect of compactness for the {S}trichartz estimates of the
  {S}chr\"odinger equations.
\newblock {\em J. Differential Equations 175}, 2 (2001), 353--392.

\bibitem{killip2015riesz}
{\sc Killip, R., Visan, M., and Zhang, X.}
\newblock Riesz transforms outside a convex obstacle.
\newblock {\em International Mathematics Research Notices 2016}, 19 (2015),
  5875--5921.

\bibitem{KiVisnaZhang16}
{\sc Killip, R., Visan, M., and Zhang, X.}
\newblock The focusing cubic {NLS} on exterior domains in three dimensions.
\newblock {\em Appl. Math. Res. Express. AMRX}, 1 (2016), 146--180.

\bibitem{KiVisZha16}
{\sc Killip, R., Visan, M., and Zhang, X.}
\newblock Quintic {NLS} in the exterior of a strictly convex obstacle.
\newblock {\em Amer. J. Math. 138}, 5 (2016), 1193--1346.

\bibitem{Kwong89}
{\sc Kwong, M.~K.}
\newblock Uniqueness of positive solutions of {$\Delta u-u+u^p=0$} in {${\bf
  R}^n$}.
\newblock {\em Arch. Rational Mech. Anal. 105}, 3 (1989), 243--266.

\bibitem{Ouss20}
{\sc Landoulsi, O.}
\newblock Blow-up solutions of the nonlinear {S}chr\"odinger equation on the
  exterior of the unit ball.
\newblock {\em preprint\/} (2020).

\bibitem{Ou19}
{\sc Landoulsi, O.}
\newblock Construction of a solitary wave solution of the nonlinear focusing
  {S}chr\"{o}dinger equation outside a strictly convex obstacle in the
  {$L^2$}-supercritical case.
\newblock {\em Discrete Contin. Dyn. Syst. A\/} (2020).

\bibitem{LaxMorawPhillip62}
{\sc Lax, P.~D., Morawetz, C.~S., and Phillips, R.~S.}
\newblock The exponential decay of solutions of the wave equation in the
  exterior of a star-shaped obstacle.
\newblock {\em Bull. Amer. Math. Soc. 68\/} (1962), 593--595.

\bibitem{LaxMorawPhillip63}
{\sc Lax, P.~D., Morawetz, C.~S., and Phillips, R.~S.}
\newblock Exponential decay of solutions of the wave equation in the exterior
  of a star-shaped obstacle.
\newblock {\em Comm. Pure Appl. Math. 16\/} (1963), 477--486.

\bibitem{DongSmithZhang2012}
{\sc {Li}, D., {Smith}, H., and {Zhang}, X.}
\newblock {Global well-posedness and scattering for defocusing energy-critical
  NLS in the exterior of balls with radial data.}
\newblock {\em {Math. Res. Lett.} 19}, 1 (2012), 213--232.

\bibitem{MeVe98}
{\sc Merle, F., and Vega, L.}
\newblock Compactness at blow-up time for {$L\sp 2$} solutions of the critical
  nonlinear {S}chr\"odinger equation in 2{D}.
\newblock {\em Internat. Math. Res. Notices}, 8 (1998), 399--425.

\bibitem{Morawetz61}
{\sc Morawetz, C.~S.}
\newblock The decay of solutions of the exterior initial-boundary value problem
  for the wave equation.
\newblock {\em Comm. Pure Appl. Math. 14\/} (1961), 561--568.

\bibitem{Morawetz62}
{\sc Morawetz, C.~S.}
\newblock The limiting amplitude principle.
\newblock {\em Comm. Pure Appl. Math. 15\/} (1962), 349--361.

\bibitem{MoraTalstonStrauss77}
{\sc Morawetz, C.~S., Ralston, J.~V., and Strauss, W.~A.}
\newblock Decay of solutions of the wave equation outside nontrapping
  obstacles.
\newblock {\em Comm. Pure Appl. Math. 30}, 4 (1977), 447--508.

\bibitem{MoraRalstonStraussCorec78}
{\sc Morawetz, C.~S., Ralston, J.~V., and Strauss, W.~A.}
\newblock Correction to: ``{D}ecay of solutions of the wave equation outside
  nontrapping obstacles'' ({C}omm. {P}ure {A}ppl. {M}ath. {\bf 30} (1977), no.
  4, 447--508).
\newblock {\em Comm. Pure Appl. Math. 31}, 6 (1978), 795.

\bibitem{NakanishiSchlag12}
{\sc Nakanishi, K., and Schlag, W.}
\newblock Global dynamics above the ground state energy for the cubic {NLS}
  equation in {3D}.
\newblock {\em Calc. Var. Partial Differential Equations 44}, 1-2 (2012),
  1--45.

\bibitem{PlVe09}
{\sc Planchon, F., and Vega, L.}
\newblock Bilinear virial identities and applications.
\newblock {\em Ann. Sci. \'Ec. Norm. Sup\'er. (4) 42}, 2 (2009), 261--290.

\bibitem{Tao06BO}
{\sc Tao, T.}
\newblock {\em Nonlinear dispersive equations}, vol.~106 of {\em CBMS Regional
  Conference Series in Mathematics}.
\newblock Published for the Conference Board of the Mathematical Sciences,
  Washington, DC, 2006.
\newblock Local and global analysis.

\bibitem{Weinstein82}
{\sc Weinstein, M.~I.}
\newblock Nonlinear {S}chr\"odinger equations and sharp interpolation
  estimates.
\newblock {\em Comm. Math. Phys. 87}, 4 (1982/83), 567--576.

\bibitem{Wilcox59}
{\sc Wilcox, C.~H.}
\newblock Spherical means and radiation conditions.
\newblock {\em Arch. Rational Mech. Anal. 3\/} (1959), 133--148.

\end{thebibliography}
\end{document}